\documentclass[final]{amsart}

\usepackage{amsmath}
\usepackage{amsfonts}
\usepackage{amssymb}
\usepackage{latexsym}
\usepackage{fancyhdr}
\usepackage{epsfig,rotating}
\usepackage{subfigure}
\usepackage{graphicx}
\usepackage{epic}
\usepackage{psfrag}
\usepackage{nicefrac}
\usepackage{bm}
\usepackage{enumerate}
\usepackage{color}
\usepackage{cite}
\usepackage{algorithm,algorithmicx,algpseudocode}
\newcommand{\uu}[1]{\mathbf{#1}}

\newcommand{\norm}[2]{\left\Vert {#1} \right\Vert_{#2}}
\newcommand{\avg}[1]{{\{\!\!\{ #1 \}\!\!\}}}

\newcommand{\cD}{{\mathcal D}}

\newcommand{\dx}{\, \mathsf{d}x}
\newcommand{\ds}{\, \mathsf{d}s}

\newcommand{\F}{\mathsf{F}}
\newcommand{\E}{\mathsf{E}}
\newcommand{\pp}{\mathsf{p}}
\newcommand{\dprod}[1]{\left<#1\right>}
\renewcommand{\u}{u^\star}

\newcommand{\T}{\mathcal{T}}
\newcommand{\refmesh}{\T_{\kappa,{\tt ref}}^{\mathcal N}}
\newcommand{\p}{\bm p}
\newcommand{\V}{\mathcal{V}(\T,\p)}
\newcommand{\Vo}{\mathcal{V}_0(\T,\p)}
\newcommand{\VD}{\mathcal{V}(\T_{\cD},\p_{\cD})}
\newcommand{\VDc}{\mathcal{V}(\T^{\prime}_{\cD},\p^{\prime}_{\cD})}
\newcommand{\VDo}{\mathcal{V}_0(\T_{\cD},\p_{\cD})}
\newcommand{\Vn}{\mathcal{V}(\T_n,\p_n)}
\newcommand{\uhp}{\u_{\rm hp}}
\newcommand{\uhpc}{u^{\star,\prime}_{\rm hp}}
\newcommand{\refspace}{{\mathcal V}(\refmesh ,\p_{\tt ref})}
\newcommand{\refspaceo}{{\mathcal V}_0(\refmesh ,\p_{\tt ref})}
\newcommand{\uref}{\u_{\kappa,{\tt ref}}}
\newcommand{\pspace}{{\mathcal V}(\T_{\kappa}^{\mathcal N},\p_{\rm p})}
\newcommand{\pspaceo}{{\mathcal V}_0(\T_{\kappa}^{\mathcal N},\p_{\rm p})}
\newcommand{\up}{\u_{\kappa,{\rm p}}}
\newcommand{\hpspace}{{\mathcal V}(\refmesh ,\p_{{\rm hp}_i})}
\newcommand{\hpspaceo}{{\mathcal V}_0(\refmesh ,\p_{{\rm hp}_i})}
\newcommand{\uhpi}{\u_{\kappa,{\rm hp}_i}}
\newcommand{\uloc}{\u_{\kappa,{\tt loc}}}
\renewcommand{\P}{\mathfrak{P}}

\newtheorem{assumption}{Assumption}

\newcommand{\myState}[1]{\State\parbox[t]{\dimexpr\linewidth-\algorithmicindent}{#1\strut}}
\newcommand{\myStateDouble}[1]{\State\parbox[t]{\dimexpr\linewidth-\algorithmicindent-\algorithmicindent}{#1\strut}}

\newtheorem{Remark}[equation]{Remark}

\newtheorem{proposition}[equation]{Proposition}

\numberwithin{equation}{section}

\numberwithin{equation}{section}

\title[$hp$--Adaptive Energy Minimisation]{Adaptive Energy Minimisation for $hp$--Finite Element Methods}

\author[P. Houston]{Paul Houston}
\address{School of Mathematical Sciences, University of Nottingham,
  University Park, Nottingham, NG7 2RD,
  UK}
\email{Paul.Houston@nottingham.ac.uk}
\thanks{PH acknowledges the financial support of the Leverhulme Trust.}

\author[T. P. Wihler]{Thomas P.~Wihler}
\address{Mathematisches Institut, Universit\"at Bern, Sidlerstrasse 5,
  CH-3012 Bern, Switzerland}
\email{wihler@math.unibe.ch}
\thanks{TW acknowledges the financial support by the Swiss National Science Foundation (SNF)}



\begin{document}

\maketitle

%%%%% Begin Abstract %%%%%%%%%%%
\begin{abstract}
This article is concerned with the numerical solution of convex 
variational problems. More precisely, we develop an iterative minimisation 
technique which allows for the successive enrichment of an underlying 
discrete approximation space in an adaptive manner. Specifically, we outline 
a new approach in the context of $hp$--adaptive finite element methods 
employed for the efficient numerical solution of linear and nonlinear second--order 
boundary value problems. Numerical
experiments are presented which highlight the practical performance of this 
new $hp$--refinement technique for both one-- and two--dimensional
problems.
\end{abstract}
%%%%% end %%%%%%%%%%%

%%%% Start %%%%%%

\section{Introduction}

Over the last few decades, tremendous progress has been made on both the mathematical analysis and practical 
application of finite element methods to a wide range of problems of industrial importance. In particular, 
significant contributions have been made in the area of {\em a posteriori} error estimation and automatic
mesh adaptation. 
For recent surveys and historical background, 
we refer to \cite{Ainsworth-Oden,BR-Acta,Johnson-et-al,sh02,Sz-B,Verfurth96},
and the references cited therein. Here, adaptive methods seek to automatically 
enrich the underlying finite element space, from which the numerical solution is sought, in order to compute efficient and reliable numerical 
approximations. The standard approach used within much of the literature is to simply undertake local isotropic refinement 
of the elements ($h$--refinement). However, in recent years,
so-called $hp$--adaptive finite element methods have been devised, whereby both local subdivision of the elements and 
local polynomial-degree-variation ($p$--refinement) are employed. These ideas date back to the work by
Babu\v{s}ka and co-workers (cf. \cite{BG88,BabuskaSuri87,BabuskaSuri89,BabuskaSuri94}); see also the
recent books \cite{De07,KS,Schwab98,opac-b1101124}, and the references cited therein.
The exploitation of general $hp$--refinement strategies 
can produce remarkably efficient methods with high algebraic or even exponential rates of convergence. Moreover, such approaches can also be 
combined with anisotropic refinement techniques in order to efficiently approximate problems with sharp transition features. 
These techniques enable the user to perform accurate and reliable computational simulations without excessive computing 
resources, and with the confidence that complex local features of the underlying solution are accurately captured.

Many physical processes can be modelled by locating critical points of a given (in our setting, convex) energy functional, over an admissible space of functions; a typical example includes quasilinear partial differential equations (PDEs). In this article we consider the application of adaptive finite element techniques, employing a combination of $hp$--mesh refinement, to problems of this type. 
In particular, we consider a new and widely applicable paradigm for adaptive mesh generation within which we directly seek to construct the $hp$--finite element space in order to approximate the critical point of the underlying energy functional associated to the problem of interest. The simplest such example is the one--dimensional Poisson equation on the interval $(0,1)$, subject to a load $f$; in this case, we seek to minimise $\E(u) = \nicefrac12 \int_0^1 u_x^2 \dx - \int_0^1fu \dx$ over an appropriate solution space $V$ (which naturally incorporates the boundary conditions). The corresponding standard Galerkin finite element approximation of this problem automatically inherits the same energy minimisation property with respect to the underlying finite element space $V_h\subset V$, i.e., the finite element solution is the unique minimiser of $\E(\cdot)$ over $V_h$. With this idea in mind, a natural approach is to adaptively modify the finite element space $V_h$ in a manner which seeks to directly decrease the energy $\E$, i.e., denoting the new finite element solution and finite element space by $u_h'$ and $V_h'$, respectively, we require that $\E(u_h') \leq \E(u_h)$. By considering an appropriately defined elementwise energy functional $\tilde{\E}^\prime_\kappa$, with $\kappa$ denoting the current element in the underlying computational mesh, we devise a competitive refinement strategy which marks elements for refinement. More precisely, in the context of our one--dimensional example, consider an $h$--refinement strategy which subdivides each element into two sub-elements. We may then compute the numerical solution to a local finite element problem posed on a local patch of elements which includes the two sub-elements. On the basis of this local reference approximation, we may then determine the predicted (elemental) energy loss if the proposed refinement (i.e., a bisection of each element into two sub-elements) is undertaken. Once the predicted energy loss has been computed for all elements in the mesh, a percentage of elements with the largest predicted energy loss may be identified and subsequently refined. This idea naturally extends to the $p$--refinement setting, whereby, additional higher--order modes are used to locally enrich the finite element solution elementwise. With this in mind, we propose a competitive $hp$--adaptive refinement strategy which computes the maximal predicted energy loss on each element based on comparing a $p$--refinement of each element (i.e., an isotropic increase of the elemental polynomial degrees by~1) with a collection of $h$--refinements of the same element (featuring different local polynomial degree distributions), which are selected so as to lead to the same increase in the number of degrees of freedom associated with the current element as the $p$--enrichment, cf. \cite{De07,Demkowicz:2002:FAH:608985.608996,opac-b1101124,ROD89}. A key aspect of this algorithm is the computation of a local elementwise/patchwise reference solution needed for the definition of $\tilde{\E}^\prime_\kappa$. In order to illustrate the key ideas, for the purposes of this article, we restrict our discussion to convex optimisation problems posed on a computational domain $\Omega \subset {\mathbb R}^d$, $d \geq 1$. However, we point out that this strategy is completely general in the sense that it can be applied to any physical problem which may be modelled as a critical point of a given energy functional $\E$ (including saddle point problems, see, e.g., ~\cite{Rabinowitz:86}). In particular, one of the key advantages of our proposed approach is that it naturally facilitates the use of $hp$--mesh adaptation, and indeed even anisotropic $hp$--mesh refinement. By considering an enrichment of the finite element space locally using any combination of isotropic/anisotropic $h$--/$p$--refinement, an element $\kappa$ can be refined according to the refinement which leads to the maximal predicted energy loss. This is in contrast to standard adaptive techniques, whereby elements are marked for refinement according to the size of a local {\em a posteriori} error indicator. Indeed, in this latter setting, such indicators rarely contain information concerning how the local finite element space should be enriched, but only indicate that a refinement should be performed. Thereby, alternative numerical techniques must be devised which are capable of determining the direction of refinement (for anisotropic refinement) or the type of refinement ($h$-- or $p$--). For the latter case, such strategies include regularity estimation \cite{eibner-melenk,HoustonSuliHPADAPT,Ma94,W10}, use of {\em a priori} knowledge \cite{bfo01,vo00}, and the computation of reference solutions within competitive refinement strategies \cite{AiSe98,De07,DoHe07,ghh-hpfem-part-II,giani_houston_compflow_2012,MelenkAPOST01,OdPaFe92,ROD89,opac-b1101124}, for example. This latter class of methods, cf. in particular \cite{De07,Demkowicz:2002:FAH:608985.608996,ROD89,opac-b1101124} and \cite{ghh-hpfem-part-II,giani_houston_compflow_2012}, are very much in the spirit of the proposed competitive refinement algorithm developed in this article. Finally, we refer to \cite{mitchell_mcclain} for an extensive review and comparison of many of the $hp$--adaptive refinement techniques proposed within the literature. 

This article is structured as follows. In Section~\ref{sc:ABSTRACT} we briefly present an abstract framework for variational problems, and consider an application to quasilinear partial differential equations. Subsequently, in Section~\ref{sc:HPFEM}, the $hp$-version finite element discretisation of such problems is presented, and a new $hp$-adaptivity approach is developed in detail. The theory will be illustrated with a number of numerical experiments on linear and quasilinear boundary value problems in Section \ref{sec-numerics}. Finally, in Section~\ref{sc:concl} we summarise the work presented in this article and draw some conclusions.

Throughout this article, we let~$L^p(D)$, $p\in[1,\infty]$, be the standard Lebesgue space on some bounded domain~$D$, with boundary $\partial D$, equipped with the norm $\|\cdot\|_{L^p(D)}$. Furthermore, for~$k\in\mathbb{N}$, we write $W^{k,p}(D)$ to signify the Sobolev space of order~$k$, endowed with the norm~$\|\cdot\|_{W^{k,p}(D)}$ and seminorm~$|\cdot|_{W^{k,p}(D)}$. For $p=2$ we write $H^k(D)$ in lieu of $W^{k,2}(D)$; moreover, $H^1_0(D)$ denotes the subspace of $H^1(D)$ of functions with zero trace on $\partial D$.

\section{Variational Problems}\label{sc:ABSTRACT}

In this section we outline an abstract framework for variational problems in Banach spaces, and consider an application to quasilinear boundary value problems.

\subsection{Abstract Minimisation Problem}\label{sc:abstract}
On a real reflexive Banach space~$X$ let us consider the minimisation
problem
\begin{equation}\label{eq:var}
\min_{u\in X} \E(u) \equiv \min_{u\in X}\left\{\F(u)-\dprod{l,u}\right\}.
\end{equation}
Here, $\dprod{\cdot,\cdot}$ is the duality product on~$X\times X'$,
where~$X'$ signifies the dual space of~$X$, and~$l\in X'$ is
given. Furthermore, throughout this manuscript, we suppose
that~$\F:\,X\to\mathbb{R}$ is a continuous and strictly convex
functional on~$X$, i.e.,
\[
\F(v_1+t(v_2-v_1))<\F(v_1)+t(\F(v_2)-\F(v_1))\qquad\forall t\in[0,1] ~~ \forall
v_1,v_2\in X.
\]
In addition, we make the assumption that~$F$ satisfies the coercivity type condition
\begin{equation}\label{eq:Fcoerc}
\F(u)-\dprod{l,u}\to+\infty\qquad\textrm{as }\|u\|_X\to\infty,
\end{equation}
where~$\|\cdot\|_X$ is a norm on~$X$. Then, \eqref{eq:var} possesses a unique
minimiser~$\u\in X$; see, e.g.,
\cite[Corollary~42.14]{Ze85}. Furthermore, the problem of finding~$\u\in X$
can be written in weak form as
\[
\dprod{\F'(\u),v}=\dprod{l,v}\qquad\forall v\in X,
\]
provided that $\F$ has a sufficiently regular G\^{a}teaux derivative~$\F'$.

\subsection{Model Problem} 

We now consider a specific application of the above abstract setting. To this end,
let $\Omega\subset {\mathbb R}^d$, $d\geq 1$, be an open bounded Lipschitz domain
with boundary $\partial\Omega$. We consider the following quasilinear
partial differential equation:
\begin{equation}\label{eq:model}
\begin{split}
-\nabla \cdot (\mu'(\nabla \u))+g'(\u) &=f,\qquad \mbox{ in } \Omega,\\
\u&=0,\qquad \mbox{ on } \partial\Omega.
\end{split}
\end{equation}
Here, $f=f(x)$, $\mu=\mu(\nabla u)$, and $g=g(u)$ are given functions,
and~$\u=\u(x)$ is the unknown analytical solution. 
We suppose that~$f\in L^q(\Omega)$, for some~$q>1$. The corresponding variational problem reads:
\begin{equation}\label{eq:varbvp}
\min_{u\in X} \E(u):= \min_{u\in X}\int_\Omega \left\{\mu(\nabla u)+g(u)-fu\right\}\dx,
\end{equation}
where~$X=W^{1,p}_0(\Omega)$ for some suitable~$p>1$.

With this notation, the following proposition holds.

\begin{proposition}
Let~$\mu$ and~$g$ from~\eqref{eq:varbvp} be strictly convex and convex, respectively, and both continuous on~$\mathbb{R}^d$ and $\mathbb{R}$, respectively. Furthermore, suppose that, for some constants~$C_1,C_2>0$, $p>1$, and~$c_1,c_2\in\mathbb{R}$ the lower bounds hold,
\begin{align}
\mu({\boldsymbol{\xi}})&\ge C_1|\boldsymbol{\xi}|^p,\label{eq:coercive}\\
g(\eta)&\ge c_1\eta+c_2,\label{eq:glow}
\end{align}
as well as the growth conditions
\begin{align}
\mu(\boldsymbol{\xi})&\le C_2(1+|\boldsymbol{\xi}|^p),\label{eq:mu1}\\
g(\eta)&\le C_2(1+|\eta|^p),\label{eq:g1}
\end{align}
for any~$\boldsymbol{\xi}\in\mathbb{R}^d$ and any $\eta\in\mathbb{R}$. Then, for any given~$f\in L^q(\Omega)$, where~$\nicefrac{1}{p}+\nicefrac{1}{q}=1$, \eqref{eq:varbvp} has a unique solution in~$X=W^{1,p}_0(\Omega)$ as well as on any linear subspace of~$X$.
\end{proposition}

\begin{proof}
Let us define
\[
u\mapsto \F(u):=\int_\Omega\left\{\mu(\nabla u)+g(u)-fu\right\}\dx.
\]
Then, we can cast~\eqref{eq:varbvp} into the abstract framework of~\eqref{eq:var}, with~$l=0$ in~$X'$. We check the conditions from Section~\ref{sc:abstract} separately. To this end, we follow the proof presented
in~\cite[Example~42.15]{Ze85}.
\begin{enumerate}[(i)]
\item {\em Continuity of~$\F$:} Let us consider a sequence~$\{u_n\}\subset W^{1,p}_0(\Omega)$, with a limit~$\overline u\in W^{1,p}_0(\Omega)$, i.e., $u_n\to\overline u$ as~$n\to\infty$. Then, with~\eqref{eq:mu1}, the Nemyckii operator $\boldsymbol{\zeta} \mapsto\mu(\boldsymbol{\zeta} )$ is continuous from~$[L^p(\Omega)]^d$ to~$L^1(\Omega)$; see~\cite[Proposition~26.6]{Ze90}. Therefore,
\[
\left|\int_\Omega \{\mu(\nabla\overline u)-\mu(\nabla u_{n})\}\dx\right|
\le \|\mu(\nabla\overline u)-\mu(\nabla u_{n})\|_{L^1(\Omega)}\to 0
\]
as~$n\to\infty$. Similarly, using~\eqref{eq:g1}, as~$n\to\infty$, we have that
\[
\left|\int_\Omega \{g(u)-g(u_{n})\}\dx\right|\to 0.
\]
The continuity of~$u\mapsto\int_\Omega fu\dx$ follows from~$f\in L^q(\Omega)$ and from H\"older's inequality. Thus, $\F$ is continuous.
\item {\em Strict convexity of~$\F$:} This simply follows from the strict convexity of~$\mu$ and the convexity of~$g$.
\item {\em Coercivity:} According to~\eqref{eq:coercive} and~\eqref{eq:glow}, we find that
\begin{align*}
\F(u)&\ge C_1|u|^p_{W^{1,p}(\Omega)}+\int_\Omega \{g(u)-fu\}\dx\\
&\ge C_1|u|^p_{W^{1,p}(\Omega)}+\int_\Omega \{c_1u+c_2-fu\}\dx\\
&\ge C_1|u|^p_{W^{1,p}(\Omega)}-|c_1|\|u\|_{L^1(\Omega)}+c_2|\Omega|-\left|\int_\Omega fu\dx\right|,
\end{align*}
where~$|\Omega|$ signifies the volume of~$\Omega$. Employing the Poincar\'e-Friedrich's 
and H\"older's inequalities, we arrive at
\begin{align*}
\F(u)&\ge C\|u\|^p_{W^{1,p}(\Omega)}-(\widetilde c_1+\|f\|_{L^q(\Omega)})\|u\|_{L^p(\Omega)}-\widetilde c_2\\
&\ge C\|u\|^p_{W^{1,p}(\Omega)}-(\widetilde c_1+\|f\|_{L^q(\Omega)})\|u\|_{W^{1,p}(\Omega)}-\widetilde c_2,
\end{align*}
for some constants~$\widetilde c_1,\widetilde c_2>0$ depending on~$\Omega$. Therefore, it follows that
\[
\E(u)=\F(u)-\dprod{l,u}\equiv \F(u)\to\infty,
\]
with~$\|u\|_{W^{1,p}(\Omega)}\to\infty$. This is the coercivity condition~\eqref{eq:Fcoerc}.
\end{enumerate}
The result now follows from~\cite[Corollary~42.14]{Ze85}.
\end{proof}

\section{$hp$-Finite Element Discretisation}\label{sc:HPFEM}
Consider now a linear subspace~$X_n\subset X$
with $\dim(X_n)=n<\infty$. Then, by our previous assumptions on~$\F$,
solving the finite dimensional convex optimisation problem $\min_{u\in X_n} \E(u)$
for the unique minimiser~$\u_n\in X_n$ results in an approximation
of~$\u\in X$ from~\eqref{eq:var} with~$\u_n\approx\u$. This is the
well-known Ritz method. Equivalently, in weak form, we may seek
$\u_n\in X_n$ such that the Galerkin formulation
\[
\dprod{\F'(\u_n),v}= \dprod{l,v}\qquad\forall v\in X_n
\]
is satisfied; cf.~\cite[\S~42.5]{Ze85}.

For the purposes of discretising our model problem~\eqref{eq:model}, we will focus on an $hp$--finite element approach. To this end, let us first introduce some notation: We let $\T=\{\kappa\}$ be a subdivision of the computational domain
$\Omega$ into disjoint open simplices such that 
$\overline{\Omega}=\bigcup_{\kappa\in\T} \overline{\kappa}$ and denote by $h_\kappa$ the
diameter of $\kappa\in\T$; i.e., $h_\kappa=\mbox{diam}(\kappa)$. 
In addition, to each element~$\kappa\in\T$ we associate a
polynomial degree~$p_\kappa$, $p_\kappa\geq 1$, and collect the
$p_\kappa$ in the polynomial degree vector~$\p=[p_\kappa:\kappa\in\T]$. With
this notation we define the $hp$--finite element spaces by
\begin{align*}
\V&=\left\{v\in H^1(\Omega):\,
v|_{\kappa}\in\mathbb{P}_{p_\kappa}(\kappa)~~\forall \kappa\in\T \right\}, \\
\Vo &= \V\cap H^1_0(\Omega),
\end{align*}
where, for~$p\ge 1$, we denote by~$\mathbb{P}_p(\kappa)$ the space of
polynomials of total degree~$p$ on~$\kappa$.

The $hp$--version finite element approximation of 
the variational formulation~\eqref{eq:var} is given by:
Find the numerical approximation~$\uhp\in\Vo$ such that
\[
\E(\uhp)=\min_{u\in \Vo} \E(u),
\]
where~$\E$ is defined in~\eqref{eq:varbvp}, or equivalently, provided that the (weak) derivatives~$\mu'$ and~$g'$ belong to~$L^1_{\rm loc}(\Omega)$, in weak form:
Find $\uhp\in\Vo$ such that
\begin{equation}\label{eq:hpFEM}
a_\Omega (\uhp,v)=\ell_\Omega (v) \qquad\forall v\in\Vo.
\end{equation}
Here,
\begin{align*}
a_\Omega (w,v) &:= \int_\Omega
\left\{\mu'(\nabla w) \cdot \nabla v+g'(w)v\right\}\dx,&&w,v\in\V, \\
\ell_\Omega (v) &:= \int_\Omega fv \dx,
&& v\in\V. 
\end{align*}
This is the $hp$--finite element discretisation of~\eqref{eq:model}.
%In a similar fashion, writing $\F$ and $\dprod{\cdot,\cdot}$ in the following elementwise manner
%$$
%\F(v) \equiv \sum_{K\in\T} \F_K(v), \qquad \dprod{l,v} \equiv \sum_{K\in\T} \dprod{l,v}_K,
%$$
%for $v\in \V$, we may write the energy functional $\E$ as follows
%$$
%\E(v) = \sum_{K\in\T} \E_K(v),
%$$
%where $\E_K(v) = \F_K(v)-\dprod{l,v}_K$.

\section{$hp$--Adaptivity}

The goal of this section is to design a procedure that generates sequences of
$hp$--adaptively refined finite element spaces in such a manner as to minimise the
error in the computed energy functional $\E$. In the context of $hp$--version finite element methods, 
the {\em local} finite element space on a given element $\kappa$, $\kappa\in\T$,
%~$\mathbb{P}_{p_\kappa}(\kappa)$, 
%for some element~$\kappa\in\T$, 
may be enriched in a number of ways. In particular,
traditional $hp$--adaptive finite element methods typically make a choice between either:
\begin{itemize}
\item $p$--refinement: The local polynomial degree~$p_\kappa$ on~$\kappa$ 
is increased by a given increment, $p_{\tt inc}$: $p_{\kappa}\leftarrow p_\kappa+p_{\tt inc}$.
Typically, a value of $p_{\tt inc}=1$ is selected.
\item $h$--refinement: The element~$\kappa$ is divided into a set of~$n_\kappa$ new 
sub-elements, such that $\overline{\kappa}=\bigcup_{i=1}^{n_{\kappa}} \overline{\kappa}_i$. Here,
 $n_{\kappa}$ will depend on both the type of element to be refined, and the type of 
 refinement employed, i.e., isotropic/anisotropic. For isotropic  refinement of a 
 triangular element $\kappa$ in two--dimensions, we have $n_{\kappa} = 4$. 
 The polynomial degree may then be inherited from the parent element~$\kappa$,
i.e., we set $p_{\kappa_i}=p_\kappa$, for $i=1,\ldots,n_{\kappa}$.
\end{itemize}

Motivated by the work presented in \cite{ROD89}, cf. also 
\cite{De07,Demkowicz:2002:FAH:608985.608996,opac-b1101124}, for example, 
in this article, 
we consider a competitive refinement strategy, whereby on each element $\kappa$ in
$\T$, we estimate the predicted reduction in the {\em local contribution} to the
energy functional $\E$ based on either employing $p$--refinement, with $p_{\tt inc}=1$,
together with a series of $hp$--refinements, which lead to the same number of degrees
of freedom as the $p$--enrichment. In contrast to standard $h$--refinement, where
the subdivided elements inherit the polynomial degree of their parent, cf. above,
in this latter case, the distribution of the polynomial degrees on
the resulting sub-elements is possibly non-uniform.

\subsection{Motivation}

The key to the forthcoming $hp$--refine\-ment strategy is to estimate the predicted
reduction in the energy functional locally on each element in the finite element
mesh $\T$. With this in mind, we must first rewrite $\E$ as the sum of local contributions
on $\T$. Given that $\E$ is simply defined as an integral over $\Omega$, then clearly, we may
write
$$
\E(v) = \sum_{\kappa\in\T} \E_\kappa(v),
$$
where $\E_\kappa$ is defined in an analogous fashion to $\E$, with the integrals over $\Omega$
being restricted to integrals over $\kappa$, $\kappa\in\T$.

However, while the above definition of the local energy functionals $\E_\kappa$ seems
entirely natural, there is no guarantee that the computed error will converge
optimally based on locally minimising $\E_\kappa$ over each $\kappa$ in $\T$. In
order to investigate this issue further and to motivate the idea proposed in this article, let us consider the following second--order linear self-adjoint partial differential equation: Find $\u$ such that
\[
-\Delta \u + \u = f \quad \mbox{in } \Omega,\qquad
\u = 0 \quad \mbox{on } \partial\Omega.
\]
Thereby, we have that 
$\mu(\nabla u) = \nicefrac{1}{2} |\nabla u|^2$ and $g(u) = \nicefrac{1}{2}\, u^2$, and $X=H^1_0(\Omega)$ (i.e., $p=q=2$).
In this setting, the (global) energy functional from~\eqref{eq:varbvp}
may be written in the form
$$
\E(u) = \frac{1}{2} a_\Omega (u,u) -\ell_\Omega (u).
$$
Moreover, we may define the associated energy norm:
$\norm{u}{\E}^2 := a_\Omega (u,u)$. Given the energy norm, exploiting the symmetry of the bilinear form $a_\Omega (\cdot,\cdot)$,
we immediately deduce the following relationship between the error in the computed 
energy functional $\E$, and the error measured in the
terms of the energy norm $\norm{\cdot}{\E}$, namely:
\[
\E(\u) - \E(\uhp) 
%= - \frac{1}{2} a_\Omega (\u-\uhp,\u-\uhp) 
= - \frac{1}{2} \norm{\u-\uhp}{\E}^2.
\]
Thereby, on a global level, reduction of the error in the energy functional
$\E$ naturally leads to a reduction in the energy norm of the error. 

In order to repeat this argument on a subset $\cD\subset\Omega$, we now suppose that
the boundary datum $g$ is given and seek $\u\in H^1(\cD)$ such that 
$\u|_{\partial\cD} = g$ and
\begin{equation}\label{eq:weak3}
a_{\cD}(\u,v)=\ell_{\cD}(v) \qquad\forall v\in H^1_0(\cD).
\end{equation}
Here, $a_{\cD}(\cdot, \cdot)$ and $\ell_{\cD}(\cdot)$ are defined in an analogous
manner to $a_{\Omega}(\cdot, \cdot)$ and $\ell_{\Omega}(\cdot)$, respectively,
with the domain of integration restricted to $\cD$.
In this case, writing $\T_{\cD}$ and $\p_{\cD}$ to denote the finite element sub-mesh and
polynomial degree distribution over $\cD$, respectively, the finite element approximation is
given by: Find $\uhp \in \VD$ such that 
$\uhp|_{\partial\cD} = \Pi g$ and
\[
a_{\cD}(\uhp,v)=\ell_{\cD}(v) \qquad\forall v\in \VDo,
\]
where $\Pi g$ denotes a piecewise polynomial approximation in $H^{\nicefrac12}(\partial\cD)$ of
the Dirichlet datum $g$. Thereby, writing $\E_\cD$ to denote the restriction of the
energy functional $\E$ over $\cD$, i.e.,
$$
\E_\cD(u) := \frac{1}{2} a_\cD (u,u) -\ell_\cD (u),
$$
we deduce the following identity:
\[
\E_\cD(\u) - \E_\cD(\uhp) 
= - \frac{1}{2} a_\cD(\u-\uhp,\u-\uhp) + a_\cD(\u,\u-\uhp)-\ell_\cD(\u-\uhp).
\]
Employing integration by parts
we deduce that
\begin{align*}
a_\cD(\u,\u-\uhp)-\ell_\cD(\u-\uhp) 
&= \int_{\partial\cD}\frac{\partial \u}{\partial {\bf n}_\cD} (\u-\uhp) \ds \\
&= \int_{\partial\cD} \mu^{\prime}(\nabla \u)\cdot {\bf n}_\cD (\u-\uhp) \ds,
\end{align*}
where ${\bf n}_\cD$ denotes the unit outward normal vector on the boundary 
$\partial\cD$ of the domain $\cD$. Thereby,
\begin{equation}
\E_\cD(\u) - \E_\cD(\uhp) 
= - \frac{1}{2} a_\cD(\u-\uhp,\u-\uhp)
 +\int_{\partial\cD} \mu^{\prime}(\nabla \u)\cdot {\bf n}_\cD (\u-\uhp) \ds.
\label{eq:error_formula}
\end{equation}
Stimulated by \eqref{eq:error_formula}, we define the {\em local} energy functional
$\tilde{\E}_\cD(\cdot)$ by
\begin{equation}
\tilde{\E}_\cD(v) := \E_\cD(v) 
- \int_{\partial\cD} \mu^{\prime}(\nabla \u)\cdot {\bf n}_\cD v \ds.
\label{eq:modified_E}
\end{equation}
With this definition, we immediately deduce the following relationship between 
the error in the local energy functional and the error measured in terms of the
local energy norm, namely,
\[
\tilde{\E}_\cD(\u) - \tilde{\E}_\cD(\uhp) 
= - \frac{1}{2} \norm{\u-\uhp}{\E,\cD}^2,
\]
where, for~$w\in H^1(\cD)$, we let $\norm{w}{\E,\cD}^2 := a_\cD(w,w)$.
Moreover, if we consider the evaluation of the above local energy functional
on each element $\kappa$, $\kappa\in\T$, then
we note the following consistency condition holds
$$
\E(v) \equiv \sum_{\kappa\in\T} \tilde{\E}_\kappa (v).
$$
Let us now write $\uhp \in \VD$ and $\uhpc \in \VDc$ to denote two finite element
approximations to \eqref{eq:weak3} based on employing the computational meshes
$\T_\cD$ and $\T_\cD^\prime$, respectively, with polynomial degree vectors
$\p_\cD$ and $\p_\cD^\prime$, respectively. Assuming the finite element space
$\VDc$ represents an enrichment of the original one $\VD$, we deduce that
the expected reduction in the error in the energy functional defined 
over $\cD$ satisfies the equality
\begin{equation}\label{eq:en_reduction}
\begin{split}
\tilde{\E}_\cD(\uhp) - \tilde{\E}_\cD(\uhpc) 
&=(\tilde{\E}_\cD(\u) - \tilde{\E}_\cD(\uhpc)) - (\tilde{\E}_\cD(\u) - \tilde{\E}_\cD(\uhp))  \\
& =  \frac{1}{2} \norm{\u-\uhp}{\E,\cD}^2 - \frac{1}{2} \norm{\u-\uhpc}{\E,\cD}^2.
\end{split}
\end{equation}
Hence, by employing the modified local definition of the energy functional
$\tilde{\E}_\cD$ defined over the subdomain $\cD$, we observe that the 
expected reduction in $\tilde{\E}_\cD$ is directly related to the 
reduction in the energy norm of the error over $\cD$. The equality
\eqref{eq:en_reduction} will form the basis of the proceeding $hp$--adaptive
refinement algorithm.

\subsection{Competitive $hp$--refinement strategy} \label{sec-hprefine}

\begin{figure}[t]
\begin{center}
\includegraphics[scale=0.25]{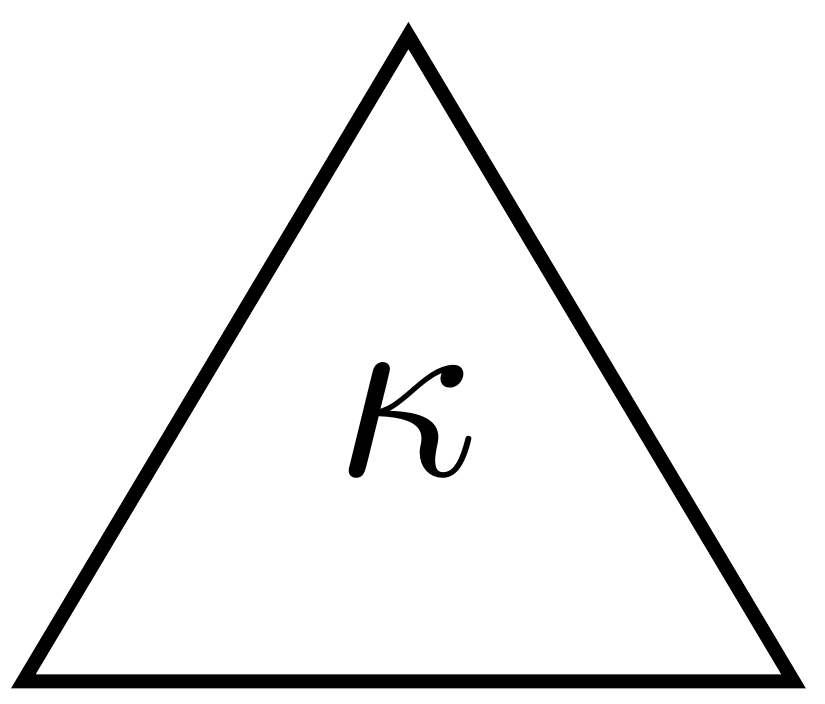} \\
(a) \\
~~\\
\begin{tabular}{cc}
\includegraphics[scale=0.25]{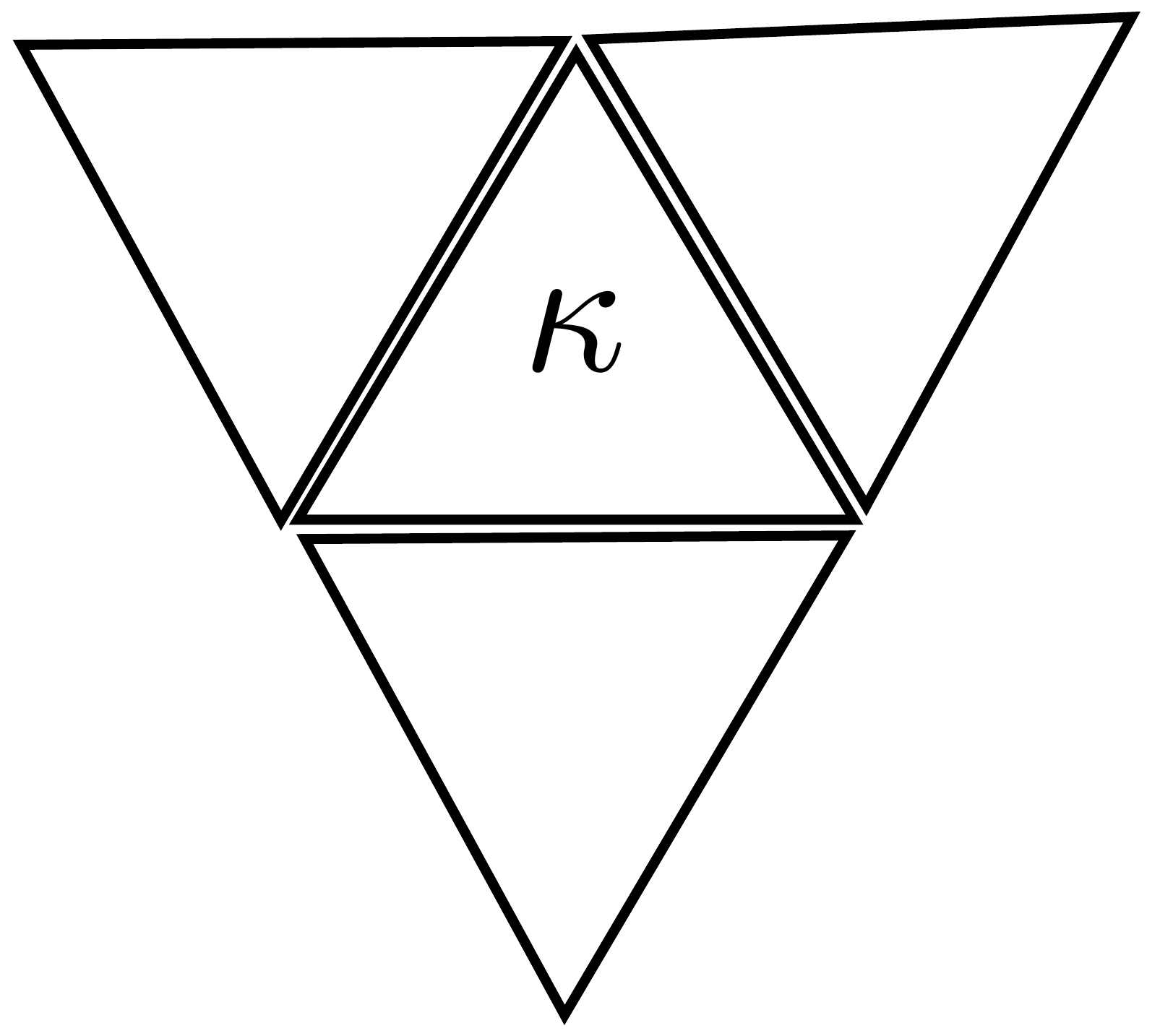} &
\includegraphics[scale=0.25]{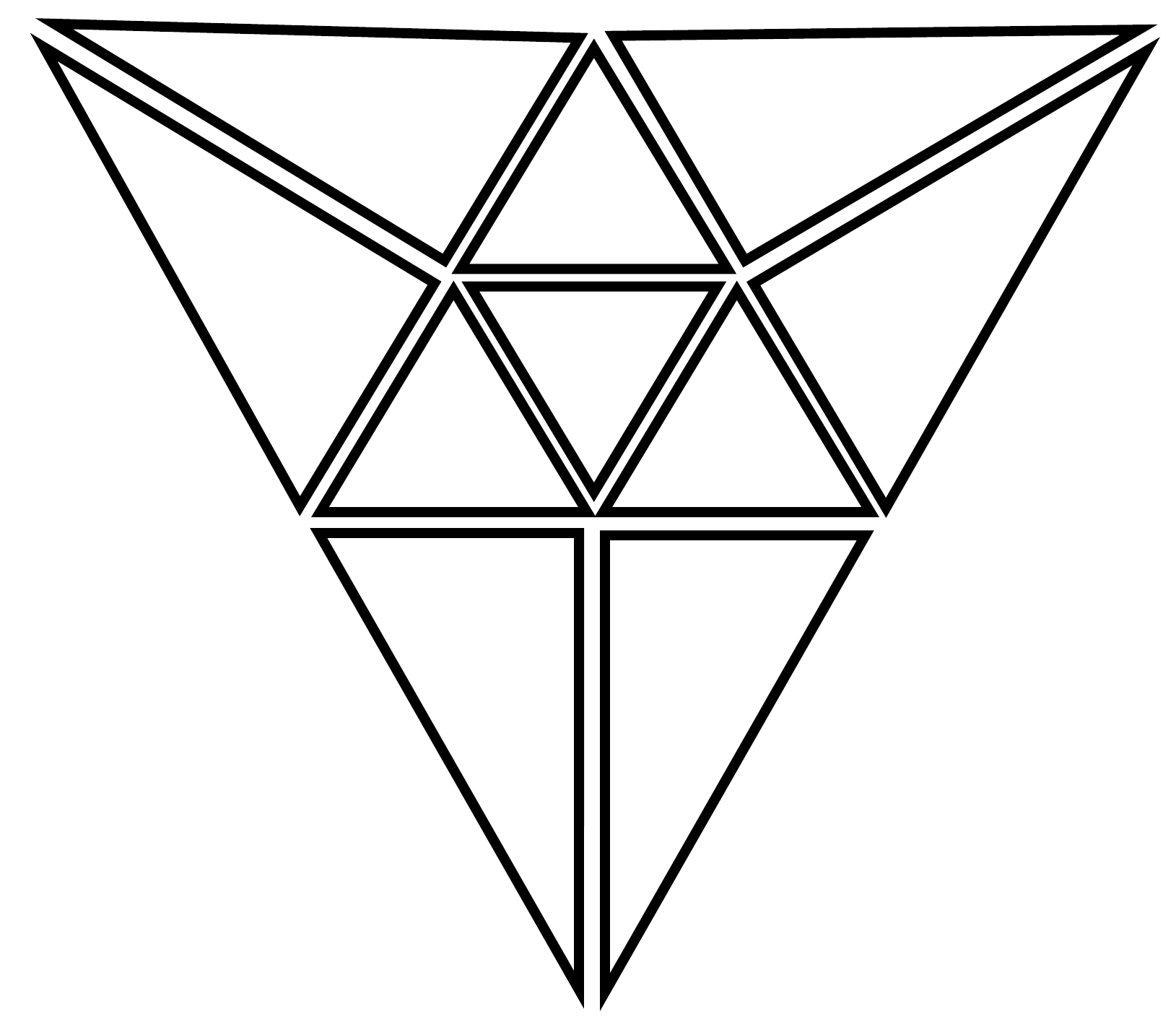} \\
(b) & (c)
\end{tabular}
\end{center}
\caption{Local element patches in two--dimensions, when triangular elements are
employed. (a) Original element $\kappa$ (assumed to be an interior
element); (b) Mesh patch $\T_{\kappa}^{\mathcal N}$, which consists of the
element $\kappa$ and its neighbours; (c) Mesh patch $\refmesh$ which is constructed
based on isotropically refining $\kappa$ (red refinement) and on a green refinement of
its neighbours.}
\label{fig:patches}
\end{figure}

In this section, we develop an $hp$--adaptivity algorithm based on employing
a competitive refinement strategy on each element $\kappa$ in the computational
mesh $\T$. The essential idea is to compute the maximal predicted
energy reduction $\tilde{\E}_\kappa(\uhp) - \tilde{\E}_\kappa(\uloc)$
on each element $\kappa\in\T$, where $\uhp$ is the (global) finite element
element solution defined by \eqref{eq:hpFEM}, and $\uloc$ is the (local)
finite element approximation to the analytical solution $\u$ evaluated on a local patch of
elements neighbouring $\kappa$, subject to a given $p$--/$hp$--refinement.
Employing the forthcoming notation $\uloc$ will either represent
$\up \in \pspace$, cf. \eqref{eq:psolution}, or $\uhpi \in \hpspace$, 
$i=1,\ldots,N_{\kappa,{\rm hp}}$, cf. \eqref{eq:hpsolution}, corresponding
to either a local $p$-- or $hp$--refinement of element $\kappa$, respectively.
Elements with the largest maximal predicted decrease in the local energy
functional are then appropriately refined. However, before we proceed,
we first note that the boundary correction term included within the
definition of the local energy functional $\tilde{E}_\kappa(\cdot)$, cf.
\eqref{eq:modified_E} with $\cD$ replaced by $\kappa$, is not computable
since it directly assumes knowledge of the unknown analytical solution $\u$.
With this in mind, we replace $\u$ by an approximate reference solution,
cf. \cite{De07,opac-b1101124}. However, in contrast to these citations, for the purposes of
the current article we simply compute local reference solutions, rather than
global ones. More precisely, given $\kappa\in\T$, we first construct the
local mesh $\T_{\kappa}^{\mathcal N}$ comprising of $\kappa$ and its immediate
face-wise neighbours, cf. Fig.~\ref{fig:patches}(b). Given $\T_{\kappa}^{\mathcal N}$,
we then uniformly (red) refine element $\kappa$ into~$n_k$ sub-elements; the introduction of any hanging
nodes may then be removed by introducing additional (green) refinements, or
alternatively, by simply uniformly refining all elements in the sub-mesh 
$\T_{\kappa}^{\mathcal N}$. For the purposes of the article, in two--dimensions,
we exploit the former strategy, purely on the basis of reducing the number
of degrees of freedom in the underlying local finite element space.
Denoting the resulting finite element mesh by $\refmesh$, cf. Fig.~\ref{fig:patches}(c),
we construct the finite element space 
$\refspace$, where 
$\p_{\tt ref}|_{\kappa^\prime} = p_\kappa+1$ for all 
$\kappa^\prime \in \refmesh$. Writing 
$\overline{\mathcal D}(\kappa) = \bigcup_{\kappa^\prime \in \refmesh} \overline{\kappa}^\prime$,
the elementwise reference solution may be computed as follows:
Find $\uref \in \refspace$ 
such that 
$\uref|_{\partial {\mathcal D}(\kappa)} = \uhp|_{\partial {\mathcal D}(\kappa)}$
and
\begin{equation}
a_{{\mathcal D}(\kappa)}(\uref,v)=\ell_{{\mathcal D}(\kappa)}(v) 
\qquad\forall v\in \refspaceo.
\label{eq:ref_solution}
\end{equation}
On the basis of the computed reference solution, we define the approximate local
energy functional on $\kappa$, $\kappa\in \T$, as follows:
\begin{equation}
\tilde{\E}^\prime_\kappa(v) := \E_\kappa(v) 
- \int_{\partial\kappa} \avg{\mu^{\prime}(\nabla \uref)} 
\cdot {\bf n}_\kappa v \ds,
\label{eq:modified_E_uref}
\end{equation}
where ${\bf n}_\kappa$ denotes the unit outward normal vector to the boundary
$\partial\kappa$ of $\kappa$, and $\avg{\cdot}$ denotes the average operator. 
More precisely, given two neighbouring elements $\kappa^+$ and $\kappa^-$, let $x$ 
be an arbitrary point on the interior face given
by $F = \partial \kappa^+ \cap \partial \kappa^-$. 
Given a vector-valued function $\uu{q}$ which is smooth inside
each element~$\kappa^\pm$, we write $\uu{q}^\pm$ to denote the traces of
$\uu{q}$ on $F$ taken from within the interior of $\kappa^\pm$,
respectively. Then, the average of $\uu{q}$ at 
$x\in F$ is given by
$
\avg{\uu{q}} =\frac{1}{2}(\uu{q}^++\uu{q}^-).
$ 
On a boundary face $F\subset \partial\Omega$, we set 
$
\avg{\uu{q}}=\uu{q}^+.
$ 

With the definition of $\tilde{\E}^\prime_\kappa(\cdot)$ given in 
\eqref{eq:modified_E_uref}, we now outline the proposed competitive refinement 
strategy on element $\kappa$, $\kappa \in \T$. Firstly, we compute the predicted
energy functional reduction when $p$--refinement is employed, i.e.,
\begin{equation}
\Delta \tilde{\E}^\prime_{\kappa,{\rm p}} := 
\tilde{\E}^\prime_\kappa(\uhp) - \tilde{\E}^\prime_\kappa(\up),
\label{eq:p_energy_red}
\end{equation}
where $\up$ is the solution of the local finite element problem:
Find $\up \in \pspace$ 
such that 
$\up|_{\partial {\mathcal D}(\kappa)} = \uhp|_{\partial {\mathcal D}(\kappa)}$
and
\begin{equation}
a_{{\mathcal D}(\kappa)}(\up,v)=\ell_{{\mathcal D}(\kappa)}(v) 
\qquad\forall v\in \pspaceo;
\label{eq:psolution}
\end{equation}
here, $\p_{\tt p}|_{\kappa^\prime} = p_\kappa+1$ for all 
$\kappa^\prime \in \T_{\kappa}^{\mathcal N}$.

\begin{figure}[t]
\begin{center}
\begin{tabular}{cc}
\includegraphics[scale=0.35]{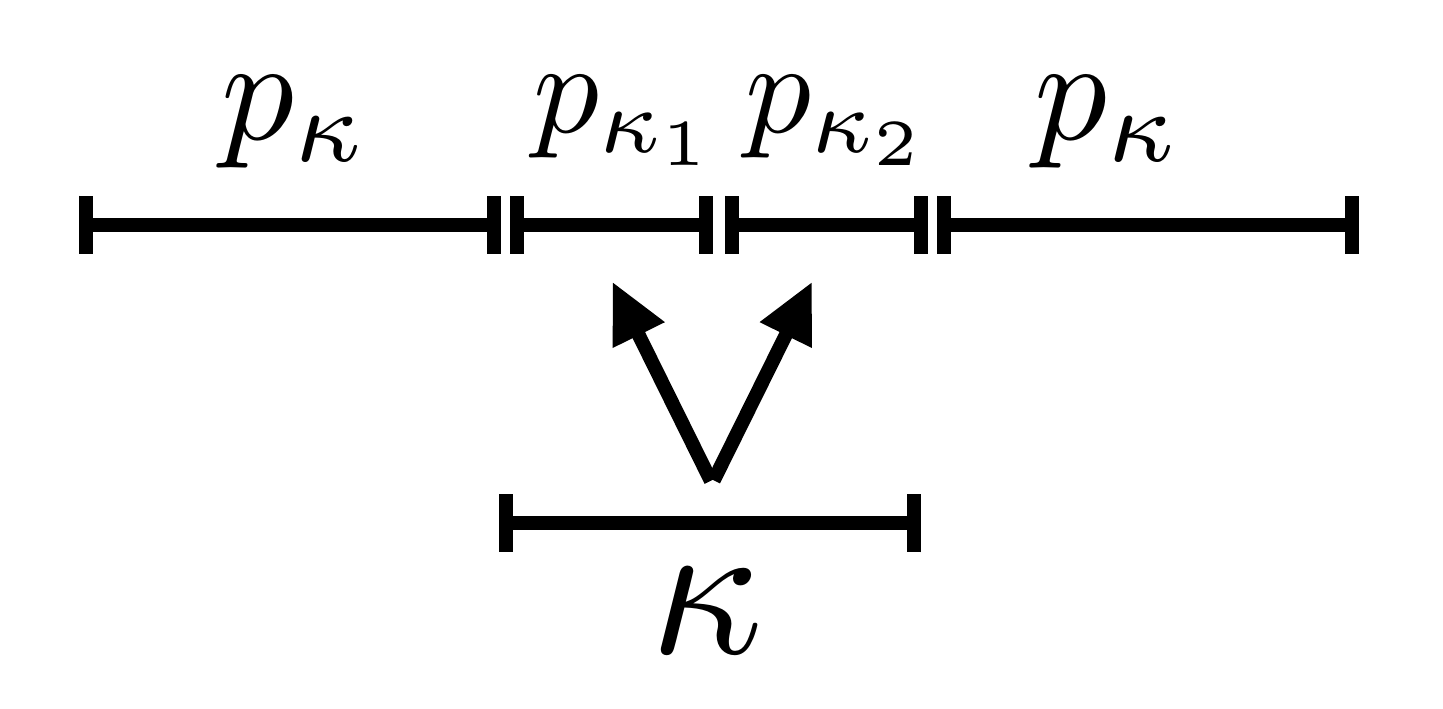} &
\includegraphics[scale=0.25]{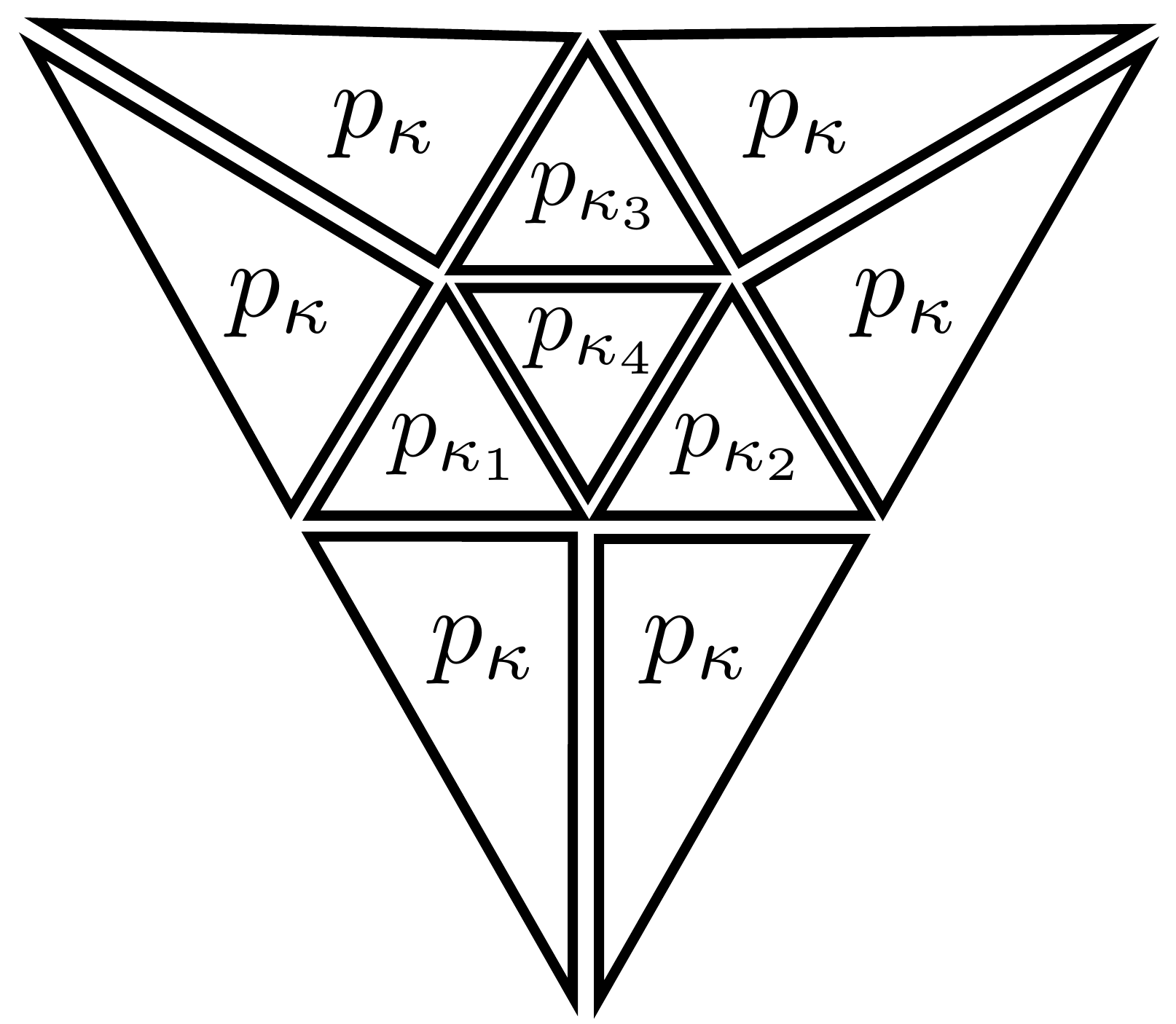} \\
(a) & (b) \\
\end{tabular}
\end{center}
\caption{Polynomial degree distribution employed for the competitive
$hp$--refinements: (a) One--dimension; (b) Two--dimensional triangular element.}
\label{fig:hp_distribution}
\end{figure}

Secondly, we also consider a sequence of competitive $hp$--refinements, such that the
number of degrees of freedom associated with the finite element space
defined over $\kappa$ is identical to the case when pure $p$--refinement has been
employed. Here, for each element $\kappa\in\T$, we again exploit the same local mesh 
$\refmesh$ employed for the computation of the 
local reference solution $\uref$. Then for the elements which result from the
isotropic refinement of $\kappa$, we employ local polynomial degrees
$p_{\kappa_i}$, $i=1,\ldots,n_\kappa$; for the remaining elements stemming from the refinement
of the neighbours of $\kappa$, we simply set the local polynomial degree equal to
$p_\kappa$, cf. Fig.~\ref{fig:hp_distribution}.
For example, in one--dimension, 
following \cite{De07,opac-b1101124}, given an element $\kappa$ with polynomial 
degree $p_\kappa$, an enrichment of $p_\kappa \rightarrow p_\kappa+1$ gives rise
to $p_\kappa+2$ degrees of freedom associated with $\kappa$. On the other hand, we
can now consider the case when $\kappa$ is uniformly subdivided into two sub-elements
$\kappa_1$ and $\kappa_2$, i.e., $n_\kappa=2$, with associated polynomial degrees $p_{\kappa_1}$ and
$p_{\kappa_2}$, respectively. To ensure that the number of degrees of freedom
in the underlying $hp$--refined finite element space defined over $\kappa_1$
and $\kappa_2$ is identical to the case when pure $p$--enrichment is 
undertaken, we require that
$$
p_{\kappa_1}+p_{\kappa_2} = p_\kappa+1.
$$
Hence, there are $N_{\kappa,{\rm hp}}=p_\kappa$, $hp$--competitive refinements and one $p$--refinement
in one--dimension.

\begin{figure}[t!]
\begin{center}
\includegraphics[scale=0.4]{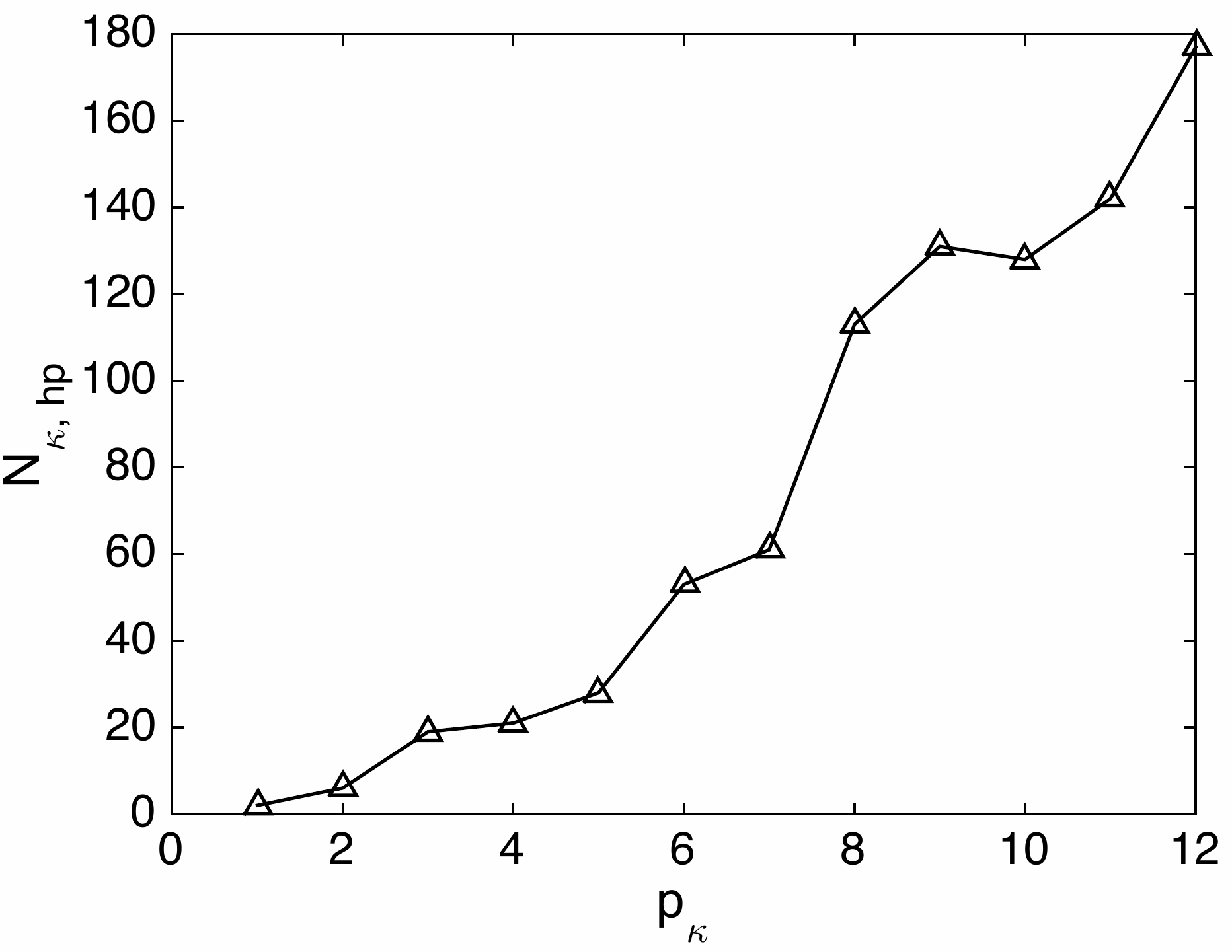}
\end{center}
\caption{Number of competitive $hp$--refinements, $N_{\kappa,{\rm hp}}$, 
versus the local polynomial degree
$p_\kappa$ when a triangular element $\kappa$ is isotropically refined.}
\label{fig:no_permutations}
\end{figure}

In higher--dimensions, the construction of the competitive $hp$--refinements is
undertaken in an analogous manner. For simplicity, we focus on the 
two--dimensional case when triangular elements are employed. 
Then for the elements which result from the
isotropic refinement of $\kappa$, we employ local polynomial degrees
$p_{\kappa_i}$, $i=1,\ldots,n_\kappa=4$; as before, the local polynomial degree of
the remaining elements stemming from the refinement
of the neighbours of $\kappa$ is set equal to
$p_\kappa$. Let us signify the set of all such polynomial degree distributions on~$\refmesh$ by~$\P_{\kappa,p_\kappa}$. Given that the full space of polynomials has been employed for the $p$--refinement, the number of degrees of freedom associated with $\kappa$ is $\nicefrac{1}{2}(p_\kappa+2)(p_\kappa+3)$. Then, for an arbitrary polynomial degree distribution $\{p_{\kappa_i}\}_{i=1}^4$ for the sub-elements $\{\kappa_i\}_{i=1}^4$ of $\kappa$, the number of degrees of freedom associated with $\kappa$ is
$$
6+\sum_{i=1}^3 \left[ \min(p_{\kappa_i},p_{\kappa_4})-1
 +2(p_{\kappa_i}-1) \right]+\frac{1}{2}\sum_{i=1}^4(p_{\kappa_i}-1)(p_{\kappa_i}-2),
$$
where we have assumed that $\kappa_4$ is the sub-element located at the interior of $\kappa$,
cf. Fig.~\ref{fig:hp_distribution}(b). Thereby, we select the set of $hp$--refinements which
satisfy the condition
$$
6+\sum_{i=1}^3 \left[ \min(p_{\kappa_i},p_{\kappa_4})-1
 +2(p_{\kappa_i}-1) \right]+\frac{1}{2}\sum_{i=1}^4(p_{\kappa_i}-1)(p_{\kappa_i}-2) =
 \frac{1}{2}(p_\kappa+2)(p_\kappa+3).
$$
Analogous expressions can also be determined for different element types, other 
kinds of refinement, e.g., anisotropic refinement, as well as in higher--dimensions.
The precise number of competitive $hp$--refinements, denoted by $N_{\kappa,{\rm hp}}$, 
is not possible to determine in
a simple closed form expression; instead, $N_{\kappa,{\rm hp}}$ 
can be precomputed for any polynomial
order. To this end, in Fig.~\ref{fig:no_permutations} we present the number of combinations of local
polynomial degrees $\{p_{\kappa_i}\}_{i=1}^4$ with respect to $p_\kappa$ in the above 
setting, i.e., for the case of isotropic refinement of a triangular element in
two--dimensions. We notice that the number~$N_{\kappa,\rm hp}$ of possible $p$-configurations is, not
 surprisingly, growing as~$p_\kappa$ increases. In view of this observation we remark that, although the subsequent local discrete
 problems defined on each corresponding (patchwise) $hp$--finite element space, 
 cf. \eqref{eq:hpsolution} below, 
 are extremely inexpensive to compute, and moreover are trivially parallelisable, from a practical 
 point of view, it might be computationally beneficial to limit the number of samples to a certain preset
 maximum~$N_{\max}$. For example, a random selection of $N_{\max}$ samples may be considered, cf.
 Section~\ref{sec-numerics} below; alternatively, a more sophisticated strategy selecting polynomial degree 
 distributions  with limited variations could be employed.

We now write $\hpspace$, 
$i=1,\ldots,N_{\kappa,{\rm hp}}$, 
to denote the finite element space based on employing the local (refined) mesh
$\refmesh$ and some local polynomial degree distribution
$\p_{{\rm hp}_i}\in\P_{\kappa,p_\kappa}$. Thereby, the following competitive $hp$--refinements may be defined:
Find $\uhpi \in \hpspace$ 
such that 
$\uhpi|_{\partial {\mathcal D}(\kappa)} = \uhp|_{\partial {\mathcal D}(\kappa)}$
and
\begin{equation}
a_{{\mathcal D}(\kappa)}(\uhpi,v)=\ell_{{\mathcal D}(\kappa)}(v) 
\qquad\forall v\in \hpspaceo,
\label{eq:hpsolution}
\end{equation}
for $i=1,\ldots,N_{\kappa,{\rm hp}}$. For each local competitive $hp$--refinement,
we compute the estimated local energy reduction
\begin{equation}
\Delta \tilde{\E}^\prime_{\kappa,{\rm hp}_i} := 
\tilde{\E}^\prime_\kappa(\uhp) - \tilde{\E}^\prime_\kappa(\uhpi),
\label{eq:hp_energy_red}
\end{equation}
for $i=1,\ldots,N_{\kappa,{\rm hp}}$. In this way, for each element $\kappa\in\T$, we may compute the maximum local predicted
error reduction 
\begin{equation}
\Delta \tilde{\E}^\prime_{\kappa,\max}
= \max \left\{ \Delta \tilde{\E}^\prime_{\kappa,{\rm p}},
\max_{i=1,\ldots,N_{\kappa,{\rm hp}}} \Delta \tilde{\E}^\prime_{\kappa,{\rm hp}_i} \right\},
\label{eq:max_error_reduction}
\end{equation}
with~$\Delta \tilde{\E}^\prime_{\kappa,{\rm p}}$ from~\eqref{eq:p_energy_red}.
Finally, we refine the set of elements $\kappa\in\T$ which satisfy the
condition
\begin{equation}
\Delta \tilde{\E}^\prime_{\kappa,\max} > \theta \max_{\kappa\in\T} 
\Delta \tilde{\E}^\prime_{\kappa,\max},
\label{eq:element_marking}
\end{equation}
where $0<\theta <1$ is a given parameter, cf. \cite{De07,opac-b1101124}. 
On the basis of \cite{De07,opac-b1101124},
throughout this article, we set $\theta = \nicefrac13$. The above competitive $hp$--refinement
strategy is summarised in Algorithm~\ref{alg:hpFEMbasic}.

\begin{algorithm}
\begin{algorithmic}[1]
\caption{Competitive $hp$-adaptive refinement procedure}
\label{alg:hpFEMbasic}
\State{Choose a coarse initial mesh~$\T_0$ of~$\Omega$ and a corresponding low-order starting polynomial degree vector~$\p_0$. Set~$n=0$.}
\State{Solve~\eqref{eq:hpFEM} for~$\uhp\in\Vn$.}
\For {each element~$\kappa\in\T_n$}
\myState{Construct the local reference mesh $\refmesh$.}
\myState{Compute the local finite element reference solution $\uref \in \refspace$ satisfying \eqref{eq:ref_solution}.}
\myState{Compute the local finite element $p$--enriched solution $\up\in\pspace$ satisfying \eqref{eq:psolution}, together with the corresponding predicted energy functional reduction $\Delta \tilde{\E}^\prime_{\kappa, {\rm p}}$, cf. \eqref{eq:p_energy_red}.}
\For {$i=1,\ldots,N_{\kappa,{\rm hp}}$}
\myStateDouble{Compute the local competitive $hp$--refined finite element solutions $\uhpi\in \hpspace$ satisfying \eqref{eq:hpsolution}, together with their respective predicted energy functional reduction $\Delta \tilde{\E}^\prime_{\kappa,{\rm hp}_i}$ defined in \eqref{eq:hp_energy_red}.}
\EndFor
\myState{Compute the maximum local predicted error reduction $\Delta \tilde{\E}^\prime_{\kappa,\max}$, cf. \eqref{eq:max_error_reduction}.}
\EndFor
\State{Determine the set of elements $\mathcal{K}_n$ which are flagged for refinement, based on the criterion \eqref{eq:element_marking}.}
\State{Perform $p$-- or $hp$--refinement on each~$\kappa\in\mathcal{K}_n$ according to which refinement takes the maximum in~\eqref{eq:max_error_reduction}. This results in a refined global finite element space~$\mathcal{V}(\T_{n+1},\p_{n+1})$.}
\State{Set $n\leftarrow n+1$, and goto Line~{\footnotesize 2}.}
\State{After sufficiently many iterations have been performed output the final solution $\uhp\in \Vn$.}
\end{algorithmic}
\end{algorithm}

\section{Numerical Examples} \label{sec-numerics}

In this section we present a series of numerical experiments to demonstrate the practical performance of the proposed
$hp$--adaptive refinement strategy outlined in Algorithm~\ref{alg:hpFEMbasic}. 

\subsection{Example 1: Linear Elliptic Problem}

In this first example, we consider a one--dimensional problem defined over the
domain $\Omega=(0,1)$. Moreover, we set
$\mu(u_x) = \nicefrac{1}{2}\,\varepsilon \, u_x^2$, $\varepsilon >0$, 
$g(u)=\nicefrac{1}{2}\, u^2$,
and $f(x) = 1$; this is equivalent to solving the linear elliptic boundary value problem:
$$
-\varepsilon \u_{xx}+\u=1, \qquad x\in \Omega,
$$
subject to homogeneous Dirichlet boundary conditions. 
We note that the analytical solution is given by
$$
\u(x) = \frac{{\rm e}^{-\nicefrac{1}{\sqrt{\varepsilon}}}-1}{{\rm e}^{\nicefrac{1}{\sqrt{\varepsilon}}}-{\rm e}^{-\nicefrac{1}{\sqrt{\varepsilon}}}}  {\rm e}^{\nicefrac{x}{\sqrt{\varepsilon}}}
         +\frac{1-{\rm e}^{\nicefrac{1}{\sqrt{\varepsilon}}}}{{\rm e}^{\nicefrac{1}{\sqrt{\varepsilon}}}-{\rm e}^{-\nicefrac{1}{\sqrt{\varepsilon}}}}  {\rm e}^{-\nicefrac{x}{\sqrt{\varepsilon}}}+1.
$$
In particular, for $0 < \varepsilon \ll 1$, the analytical solution $\u$ contains boundary layers in the vicinity of $x=0$ and $x=1$,
cf. \cite{W10}; as in \cite{W10}, we set $\varepsilon = 10^{-5}$.

\begin{figure}[t]
\begin{center}
\begin{tabular}{cc}
\includegraphics[scale=0.33]{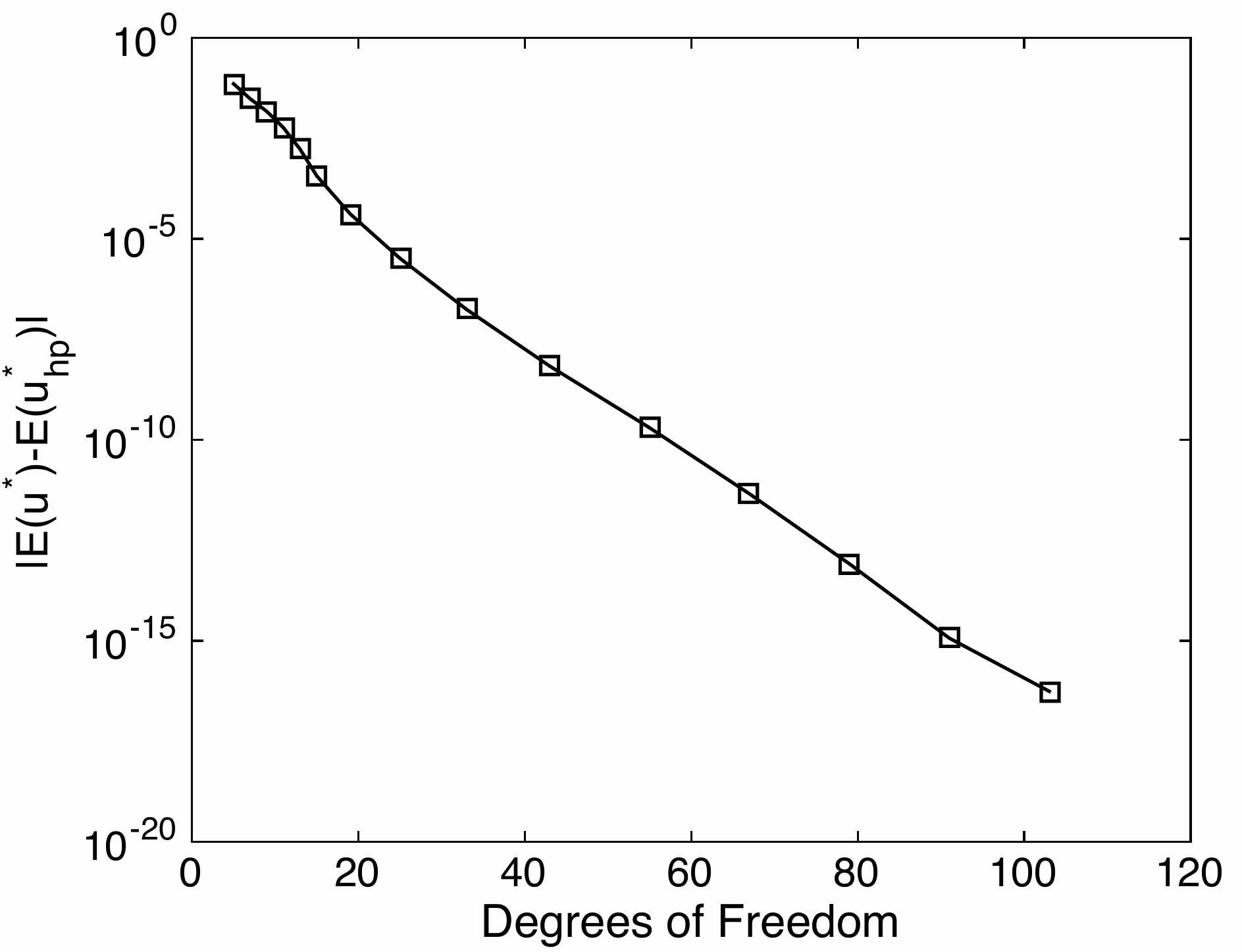} &
\includegraphics[scale=0.33]{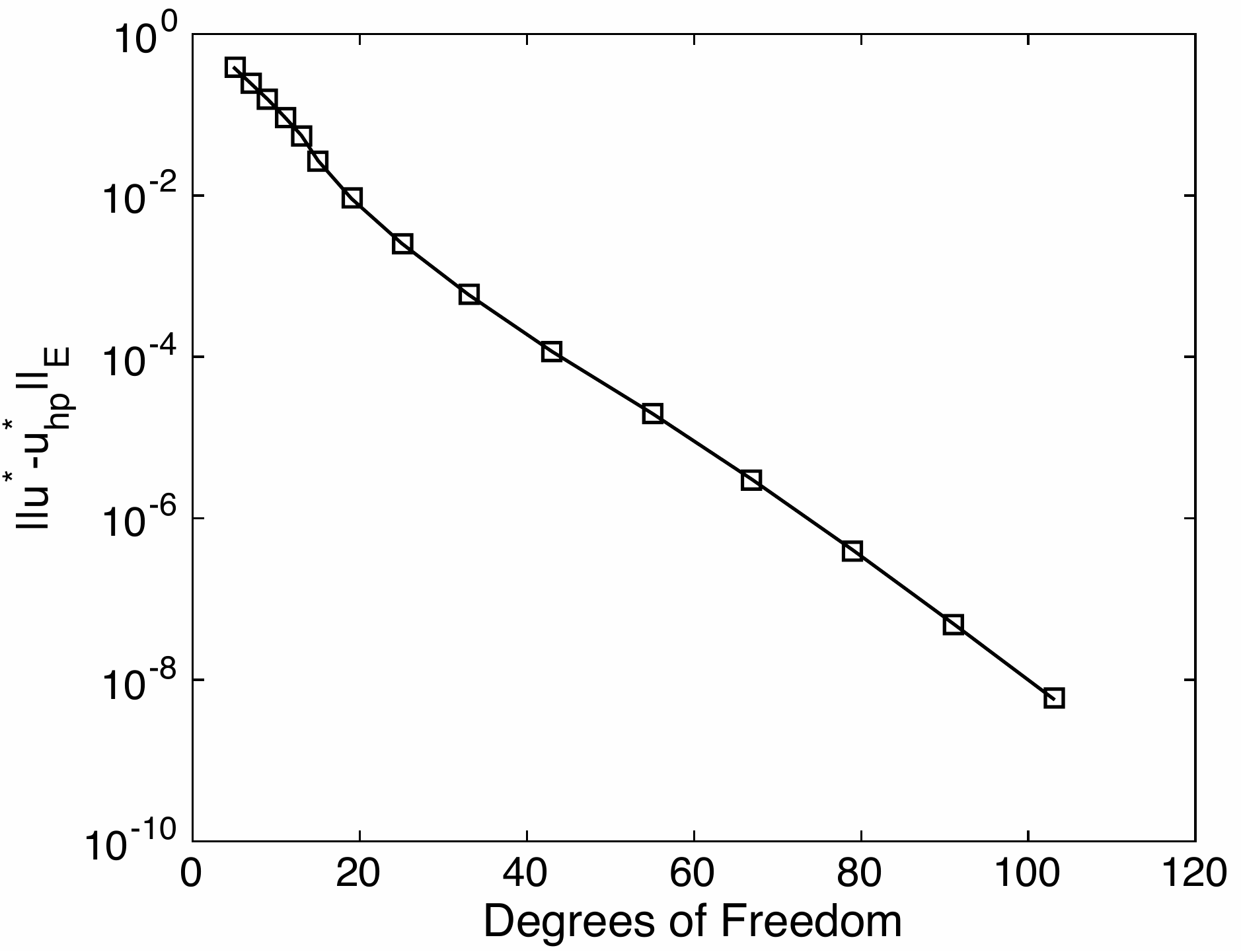} \\
(a) & (b)\\
\includegraphics[scale=0.33]{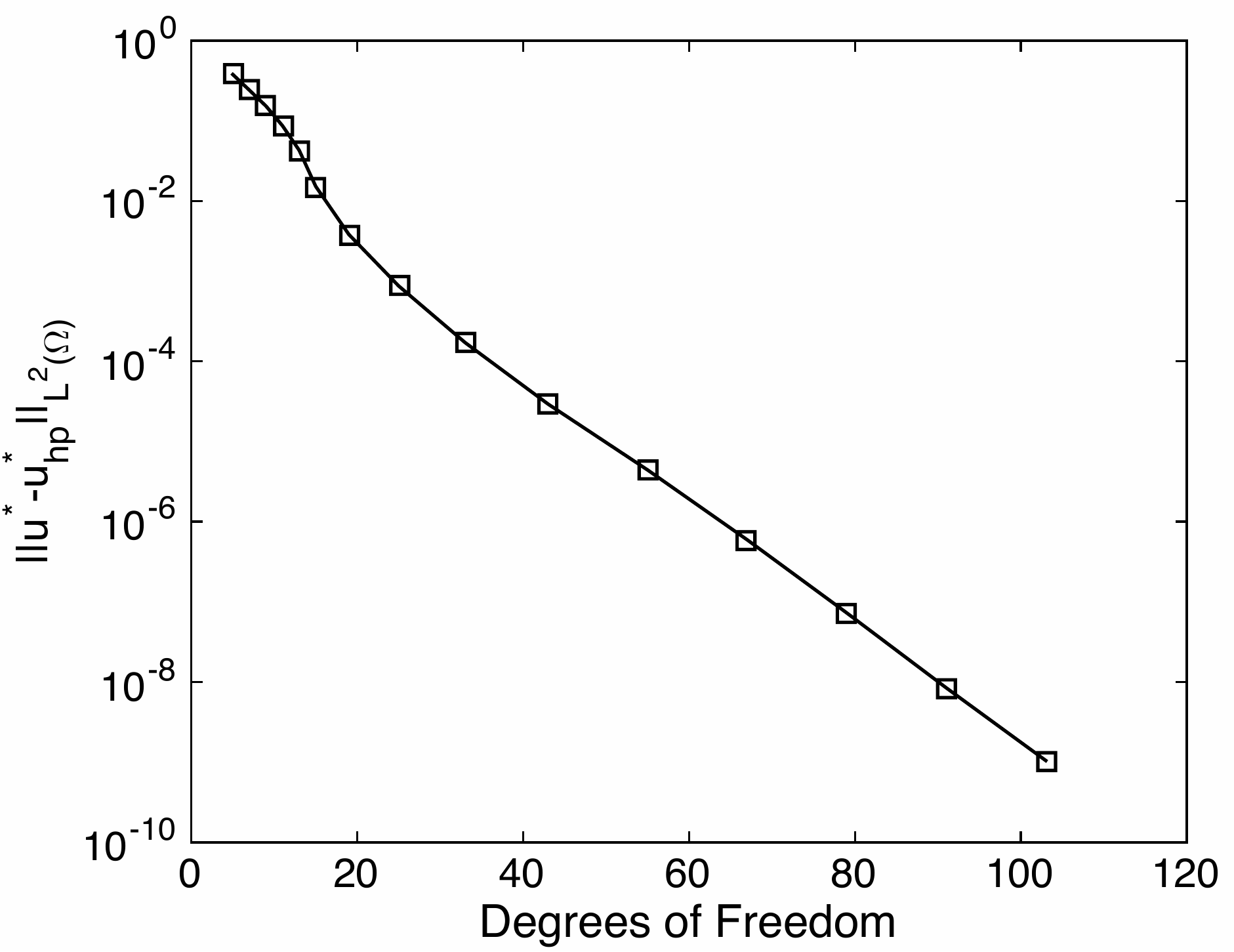} &
\includegraphics[scale=0.33]{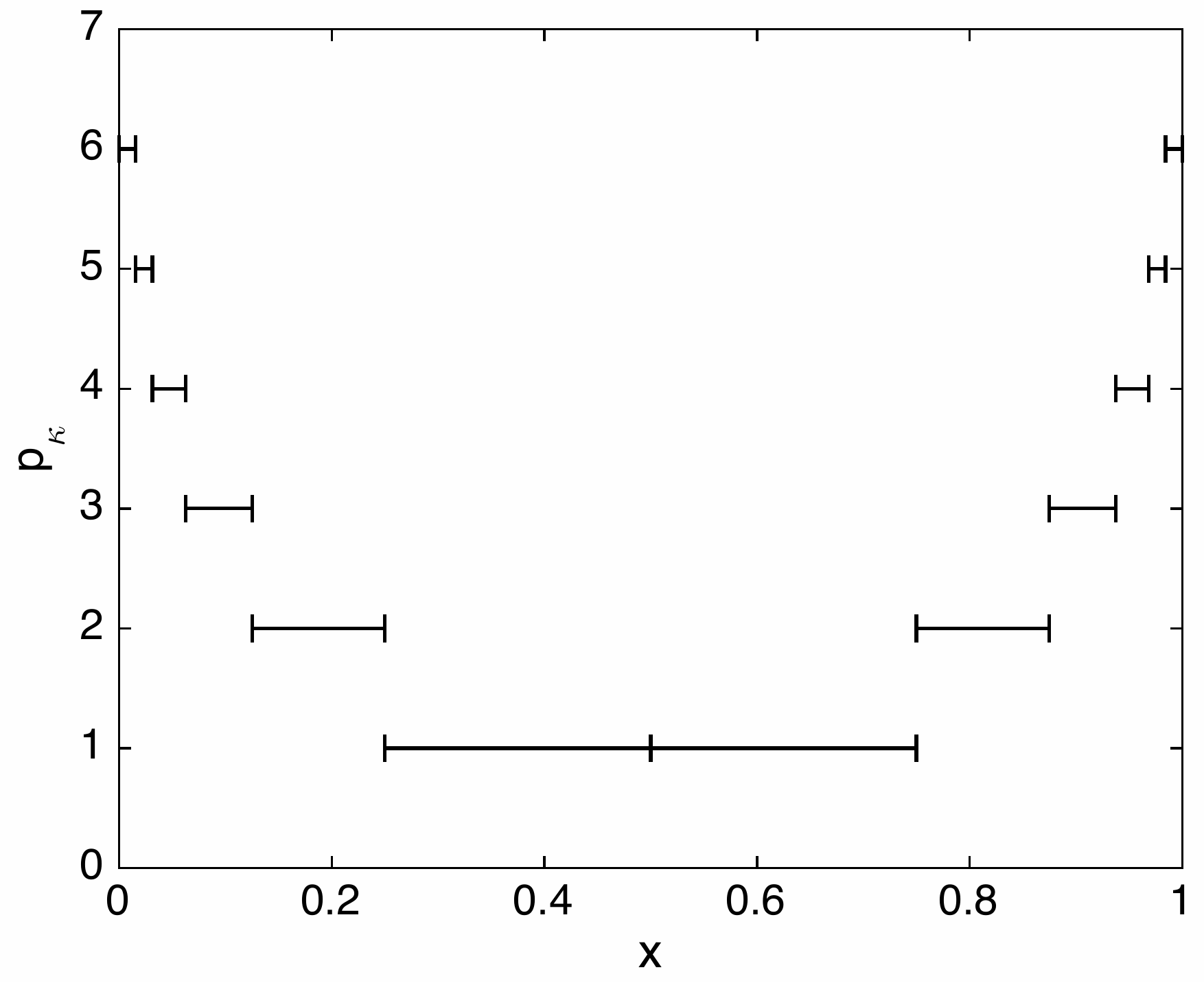} \\
(c) & (d)\\
\end{tabular}
\end{center}
\caption{Example 1. Comparison of the error with respect to the number of degrees of freedom:
(a) $|\E(\u)-\E(\uhp)|$; (b) $\|\u-\uhp \|_\E$; (c) $\|\u-\uhp\|_{L^2(\Omega)}$. 
(d) $hp$--Mesh distribution after 9 adaptive refinements. }
\label{ex1-fig1}
\end{figure}

In Fig.~\ref{ex1-fig1} we illustrate the performance of the proposed $hp$--adaptive algorithm, cf. Algorithm~\ref{alg:hpFEMbasic},
based on a starting mesh consisting of 4 elements, with the initial polynomial degree $\p= [1,1,1,1]$. Here, we have plotted
the error in the underlying energy functional~$\E$, together with the energy norm $\|\cdot\|_{\E}$ and $L^2(\Omega)$ norm
of the error, with respect to the total number of degrees of freedom employed within the finite element space
$\V$, on a linear--log scale; here, 
$$
\|v\|_{\E}^2 = \int_0^1 (\varepsilon v_x^2 + v^2) \dx.
$$
From Fig.~\ref{ex1-fig1}(a), (b), \& (c), we observe, that after an initial transient, the convergence lines for each error
measure become (on average) straight, thereby indicating exponential convergence of the quantities
$|\E(\u)-\E(\uhp)|$, $\|\u-\uhp\|_{\E}$, and $\|\u-\uhp\|_{L^2(\Omega)}$, respectively, as $\V$ is adaptively
enriched. Finally, in Fig.~\ref{ex1-fig1}(d) we show the $hp$--mesh distribution after 9 adaptive refinements. Here,
we observe that the algorithm clearly identifies the location of the boundary layers present in the analytical solution
$\u$; indeed, in these regions, local subdivision of the mesh has first been employed, followed by subsequent $p$--enrichment,
cf. \cite{W10}.

\subsection{Example 2: Strongly monotone quasilinear PDE}

\begin{figure}[t]
\begin{center}
\begin{tabular}{cc}
\includegraphics[scale=0.33]{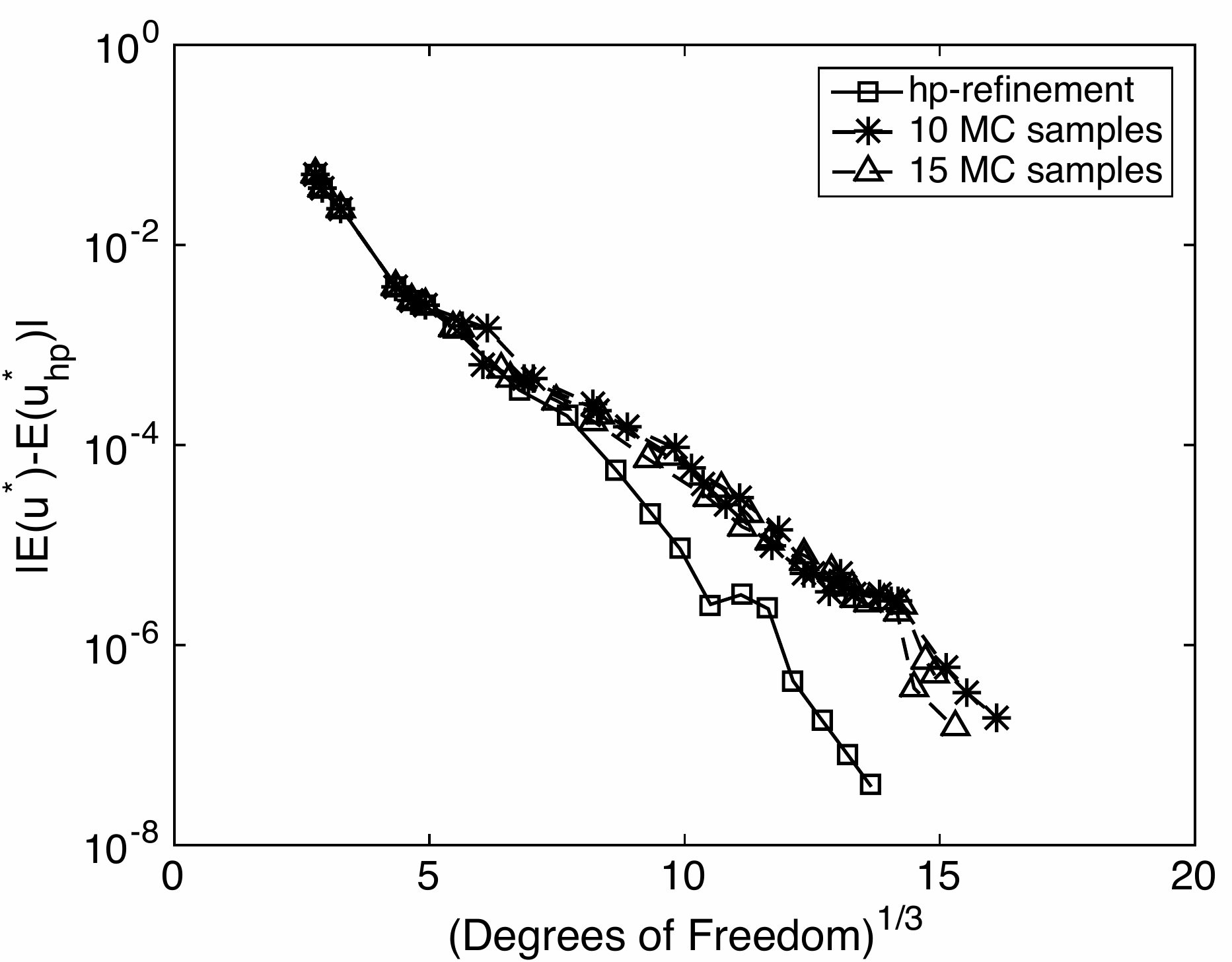} &
\includegraphics[scale=0.33]{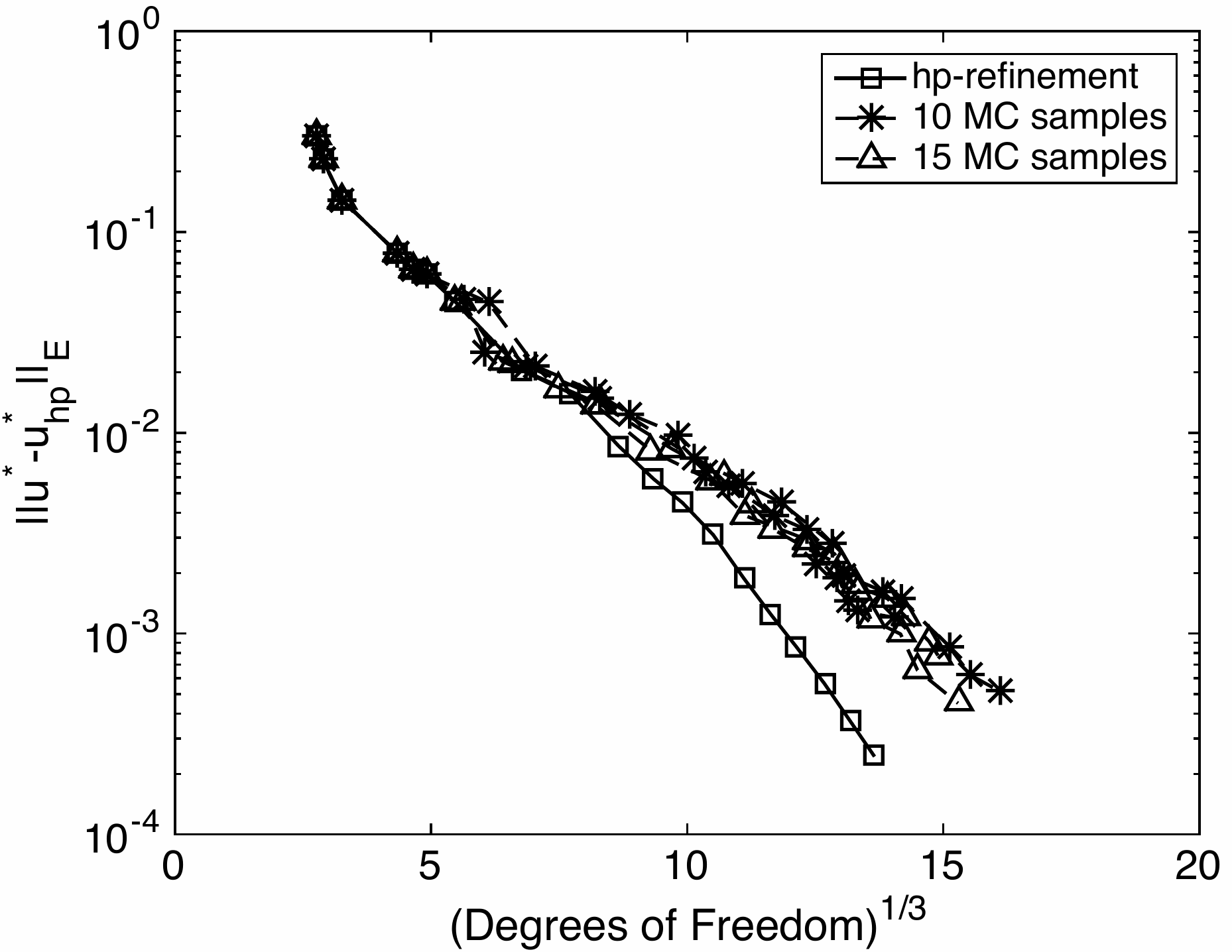} \\
(a) & (b)\\
\multicolumn{2}{c}{\includegraphics[scale=0.33]{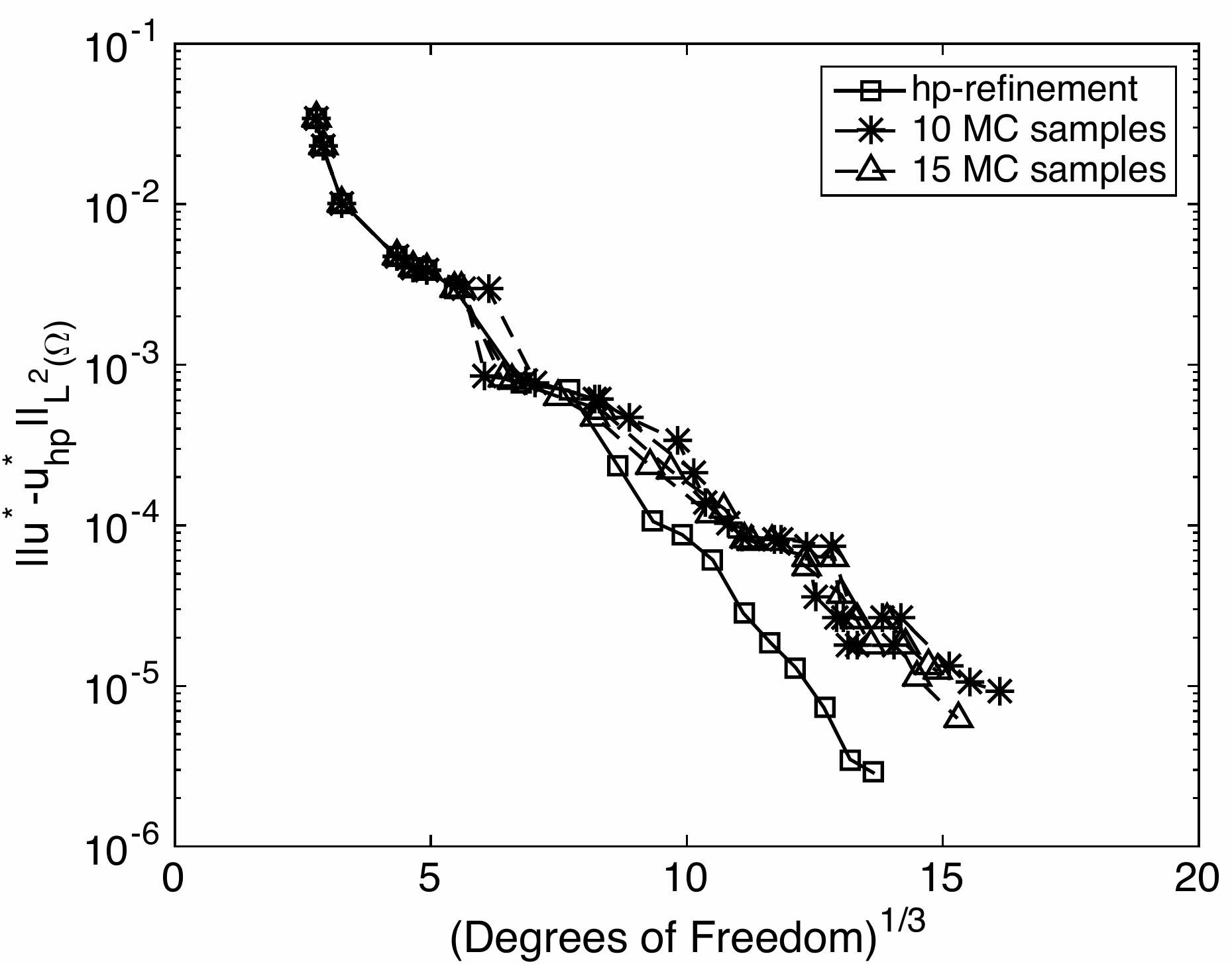}} \\
\multicolumn{2}{c}{(c)} \\
\end{tabular}
\end{center}
\caption{Example 2. Comparison of the error with respect to the third root
of the number of degrees of freedom:
(a) $|\E(\u)-\E(\uhp)|$; (b) $\|\u-\uhp \|_\E$; (c) $\|\u-\uhp\|_{L^2(\Omega)}$.}
\label{ex2-fig1}
\end{figure}

In this second example, we let $\Omega$ be the L-shaped domain
$(-1,1)^2\setminus [0,1)\times (-1,0]$, and set
$$
\mu(\nabla u) = \frac{1}{2} \left(|\nabla u|^2 - {\rm e}^{-|\nabla u|^2} \right).
$$
Thereby, the corresponding Euler--Lagrange equation for the underlying
minimisation problem corresponds to the strongly monotone quasilinear
PDE given by:
\begin{align}
-\nabla \cdot 
\left( \left(1 + {\rm e}^{-|\nabla \u|^2} \right) \nabla \u \right) =& f, 
\quad \mbox{in } \Omega. \label{eq:monotone_pde1}
\end{align}
We select $f$ and appropriate inhomogeneous Dirichlet boundary conditions 
so that the analytical solution to \eqref{eq:monotone_pde1} is given by
$$
u=r^{\nicefrac23}\sin\left( \tfrac{2}{3}\varphi\right),
$$
where $(r,\varphi)$ denote the system of polar coordinates,
cf. \cite{Congreve_et_al_2011,Houston_QL}, for example.

Selecting the energy norm $\|\cdot\|_{\E}$ to be the standard $H^1(\Omega)$ norm, 
in Fig.~\ref{ex2-fig1} we again present the convergence history of the error
in the computed energy functional~$\E$, together with $\|\u-\uhp\|_{\E}$, and $\|\u-\uhp\|_{L^2(\Omega)}$,
as the finite element space is $hp$--adaptively refined. On a linear--log scale 
(where the horizontal axis measures the third root of the total number of the degrees 
of freedom, cf.~\cite{Schwab98}), we again
observe exponential rates of convergence, in the sense that asymptotically the convergence
lines become roughly straight. In addition, in Fig.~\ref{ex2-fig1} we also 
present analogous results in the case when a Monte Carlo (MC) approach is employed to limit
the number~$N_{\max}$ of $hp$--refinement samples considered on each element. More precisely,
we randomly select samples based on employing $N_{\max} = 10$ and $N_{\max} = 15$; in each case
two typical realisations are presented. Here, we observe a slight degradation of
the rate of convergence in each of the above error quantities as our $hp$--refinement
procedure progresses, as we would expect; however, in each case exponential convergence
is retained when this simple selection principle is exploited. As noted in 
Section~\ref{sec-hprefine} more sophisticated selection principles may also be employed.

The final $hp$--mesh distribution is depicted in
Fig.~\ref{ex2-fig2}; here, we see that the computational mesh has been largely refined in the vicinity 
of the re-entrant corner located at the origin. In addition, we see that the polynomial degrees have been 
increased away from the origin, since the underlying analytical solution is smooth in this region.
In particular, we observe that the refinement algorithm has generated an $hp$--mesh distribution which 
is symmetric with respect to the line $x_2=-x_1$.

\begin{figure}[t]
\begin{center}
%\begin{tabular}{cc}
\includegraphics[scale=0.42]{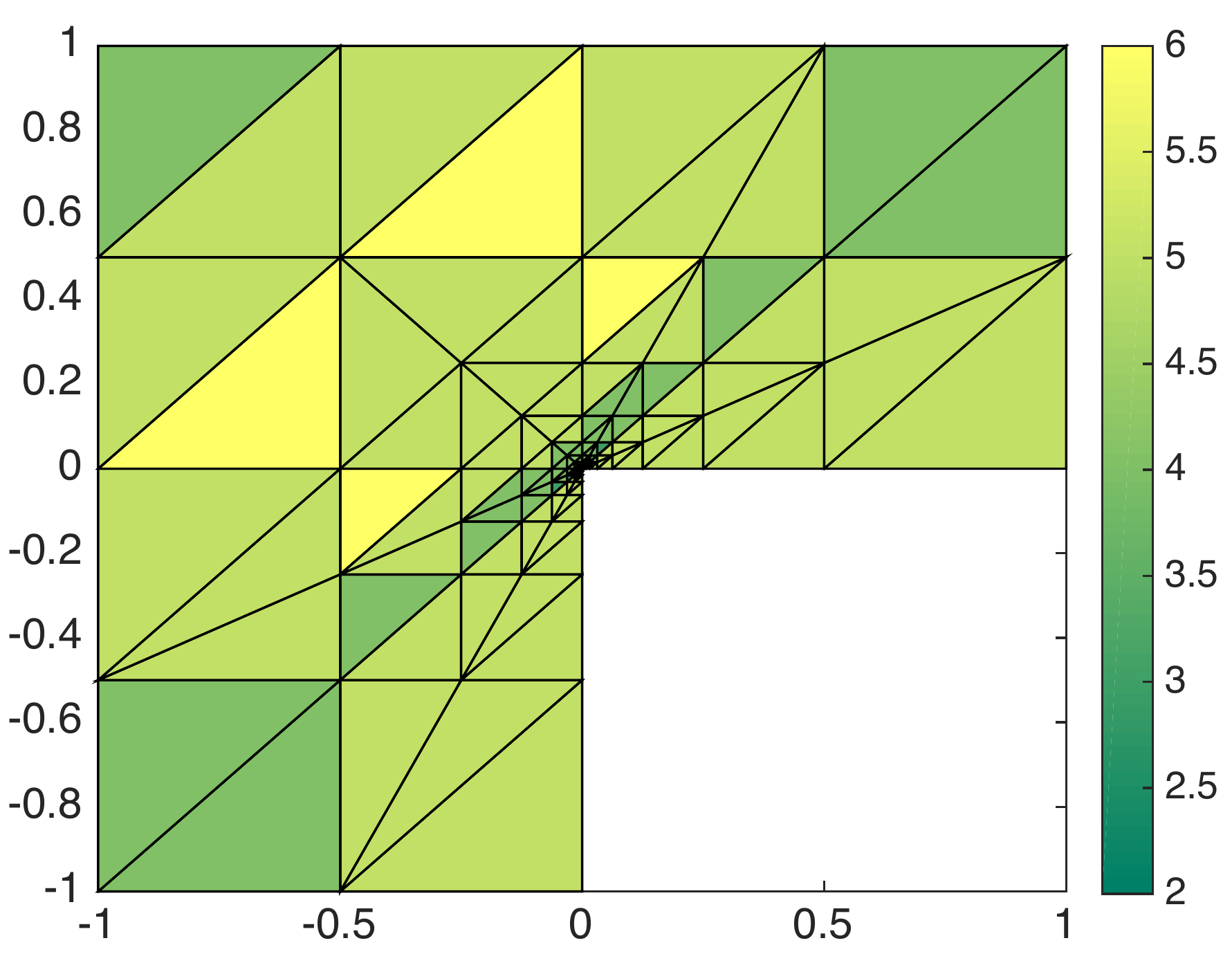} \\
(a) \\
\includegraphics[scale=0.42]{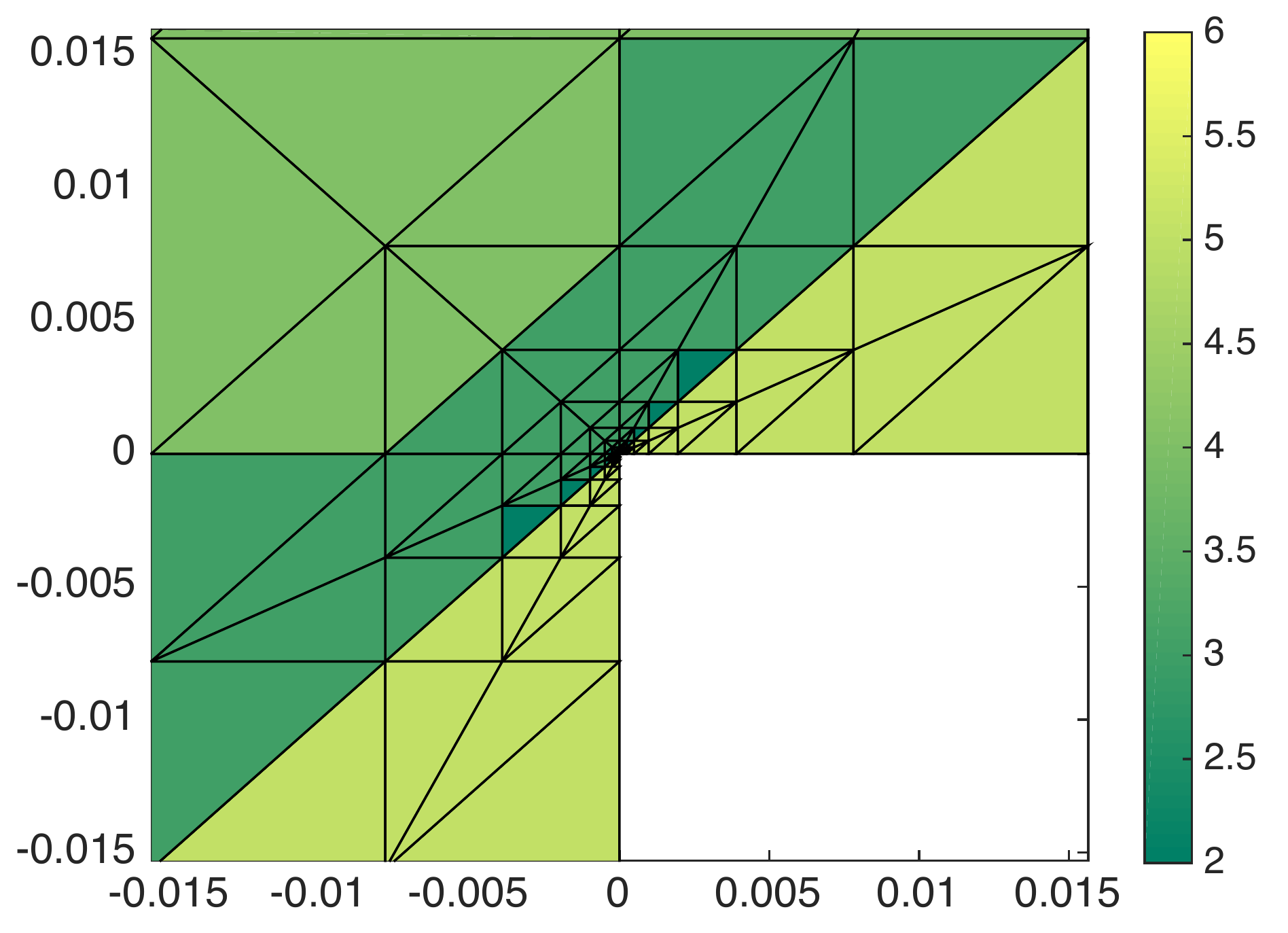} \\
(b) \\
%\end{tabular}
\end{center}
\caption{Example 2. (a) $hp$--Mesh distribution after 18 adaptive refinements;
(b) Zoom of (a).}
\label{ex2-fig2}
\end{figure}

\subsection{Example 3: $\pp$--Laplacian}

\begin{figure}[t]
\begin{center}
\begin{tabular}{cc}
\includegraphics[scale=0.33]{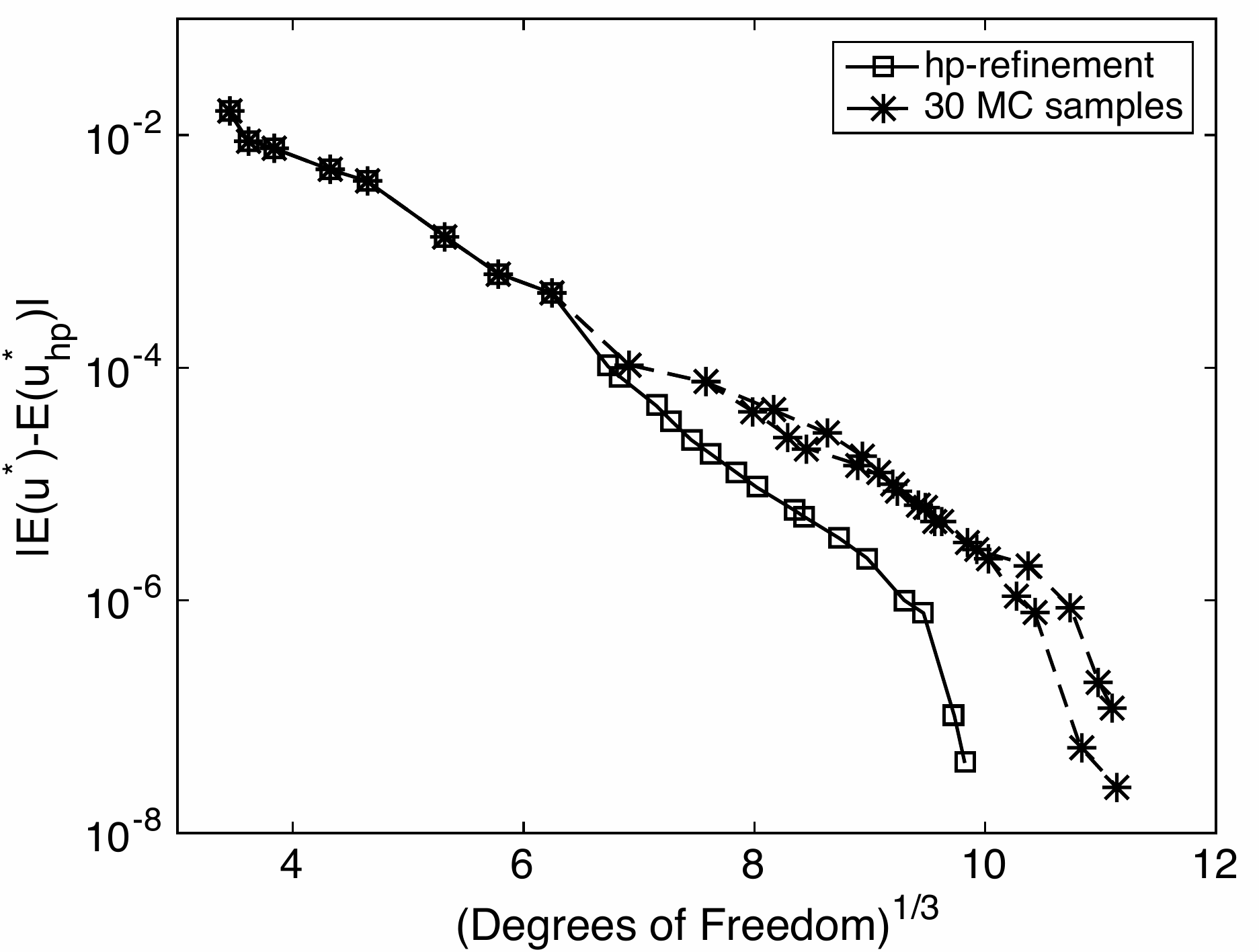} &
\includegraphics[scale=0.33]{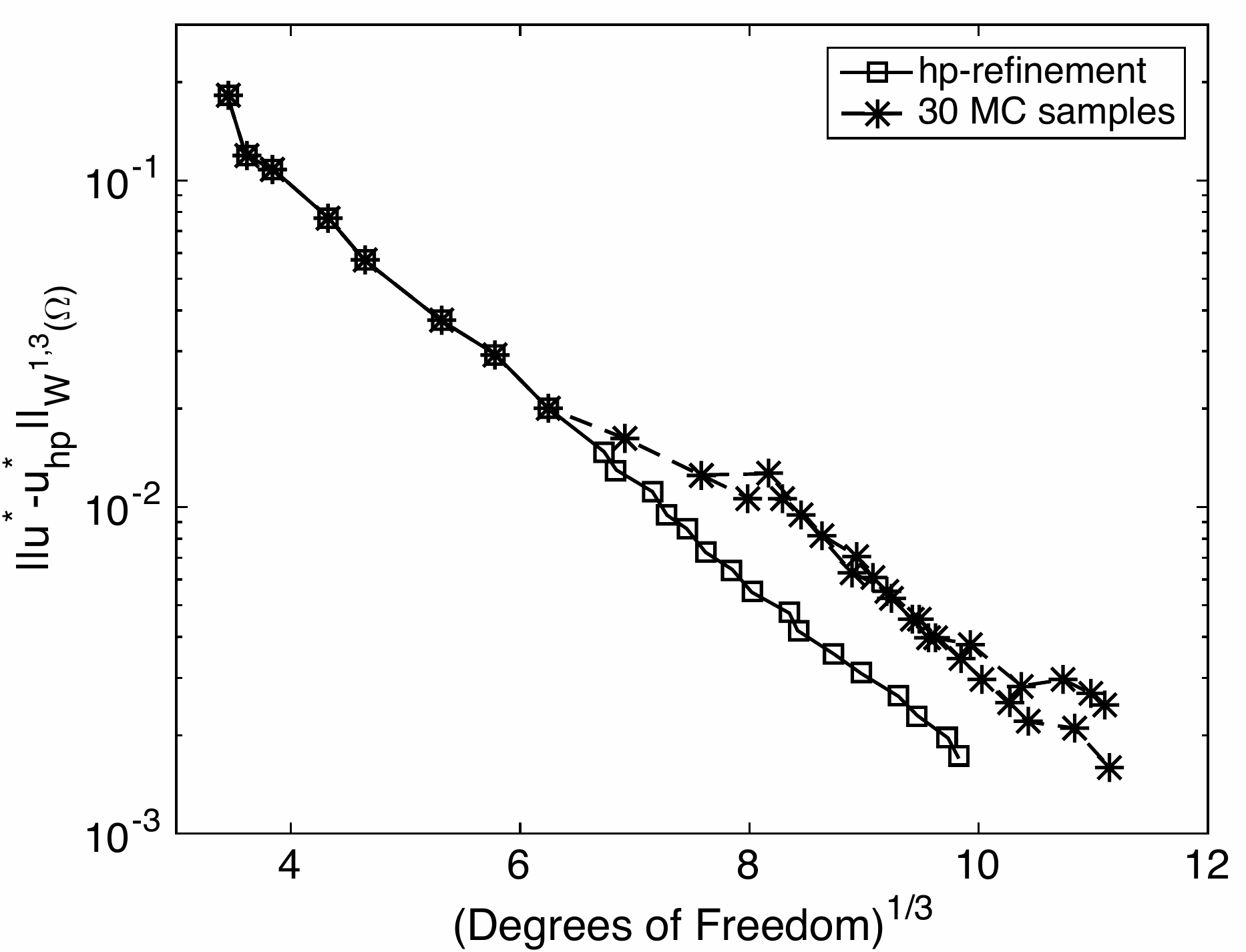} \\
(a) & (b)\\
\multicolumn{2}{c}{\includegraphics[scale=0.33]{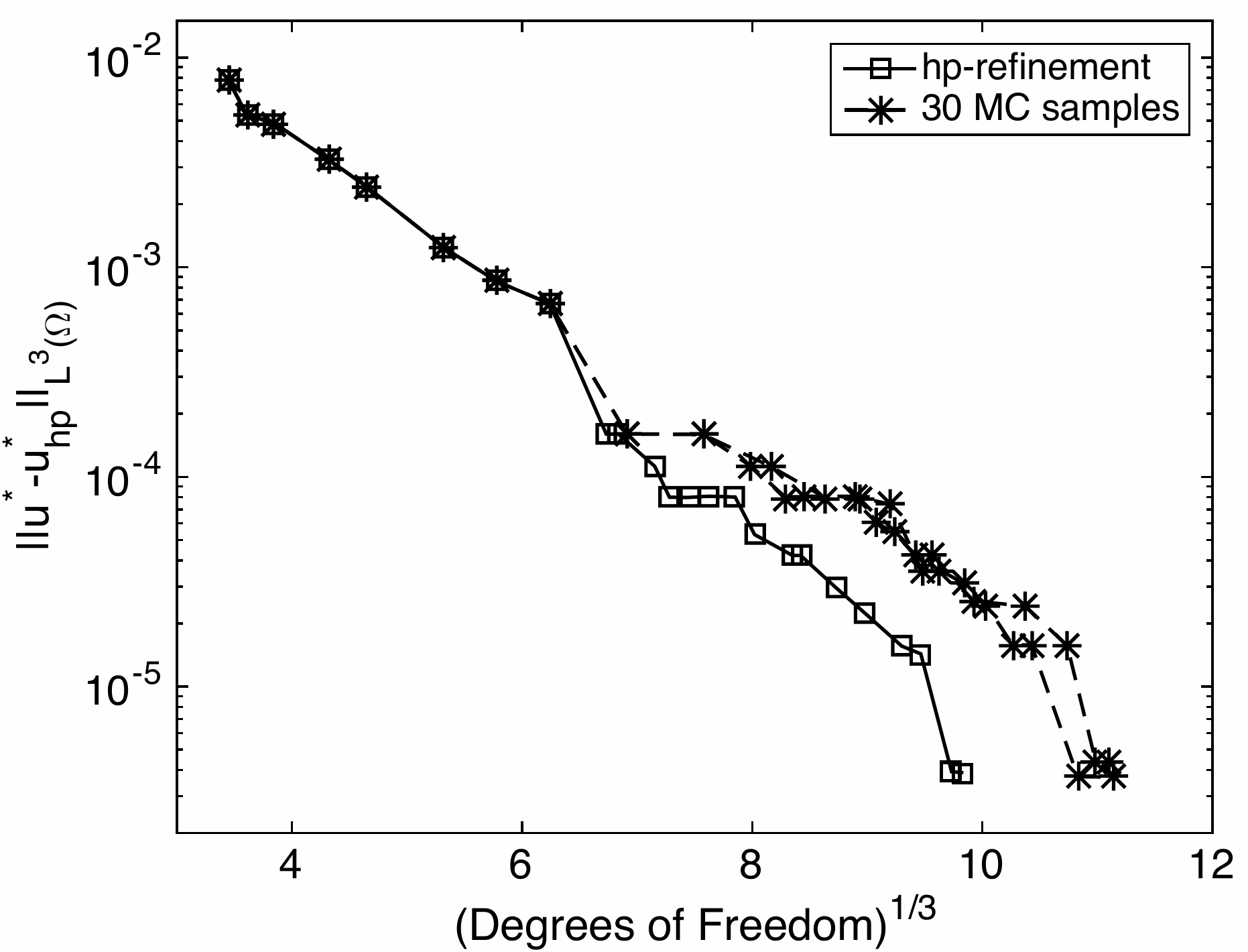}} \\
\multicolumn{2}{c}{(c)} \\\end{tabular}
\end{center}
\caption{Example 3. Comparison of the error with respect to the third root of the
number of degrees of freedom:
(a) $|\E(\u)-\E(\uhp)|$; (b) $|\u-\uhp |_{W^{1,3}(\Omega)}$; (c) $\|\u-\uhp\|_{L^3(\Omega)}$.}
\label{ex3-fig1}
\end{figure}

In this final example, for $\pp>1$, we consider the $\pp$--Laplacian problem
\begin{equation}
-\nabla \cdot (|\nabla \u|^{\pp-2} \nabla \u) = f,  \qquad \mbox{in } \Omega = (0,1)^2, \label{p-laplacian}
\end{equation}
subject to inhomogeneous Dirichlet boundary conditions. We point out that in
this setting, \eqref{p-laplacian} corresponds to the Euler--Lagrange equation
for the energy minimisation problem
$$
\min_{u\in W^{1,\pp}_0(\Omega)} \left\{ \frac{1}{\pp} \int_\Omega |\nabla u|^{\pp} \dx-\int_\Omega fu \dx \right\};
$$
i.e., we have $\mu(\nabla u) = \nicefrac{1}{\pp}  |\nabla u|^{\pp}$ and $g = 0$. We select $f$,
and impose suitable inhomogeneous Dirichlet boundary conditions, so that the analytical solution
of \eqref{p-laplacian} is given by
$$
\u(x) = r^\alpha, \qquad \alpha > 0.
$$
As in \cite{Ainsworth1999}, throughout this section, we set $\pp=3$ and $\alpha=\nicefrac34$,
which implies that $\u \in W^{\beta-\epsilon,3}(\Omega)$, where $\beta = \nicefrac{13}{6}$
and $\epsilon>0$ is arbitrarily small.

In Fig.~\ref{ex3-fig1} we plot $|\E(\u)-\E(\uhp)|$, $ \|\u-\uhp \|_{W^{1,3}(\Omega)}$,
and $\|\u-\uhp \|_{L^{3}(\Omega)}$, with respect to the third root of the number of degrees of freedom
in $\V$. As in the previous examples, we again observe exponential convergence of each of the
above error measures, as the finite element space is $hp$--adaptively modified. 
Here, we also consider the case when $N_{\max}=30$ random samples are selected; as in the previous
example, we again see that exponential convergence of each of the above error quantities is
retained, though the rate of convergence is inferior when compared to the case when all potential trial
$hp$--refinements are considered.
The final $hp$--mesh distribution is shown in Fig.~\ref{ex3-fig2}; as in the previous
examples, the adaptive algorithm clearly identifies the location of the singularity present 
within the analytical
solution $\u$, whereby $h$--refinement is undertaken.

\begin{figure}[t]
\begin{center}
%\begin{tabular}{cc}
\includegraphics[scale=0.42]{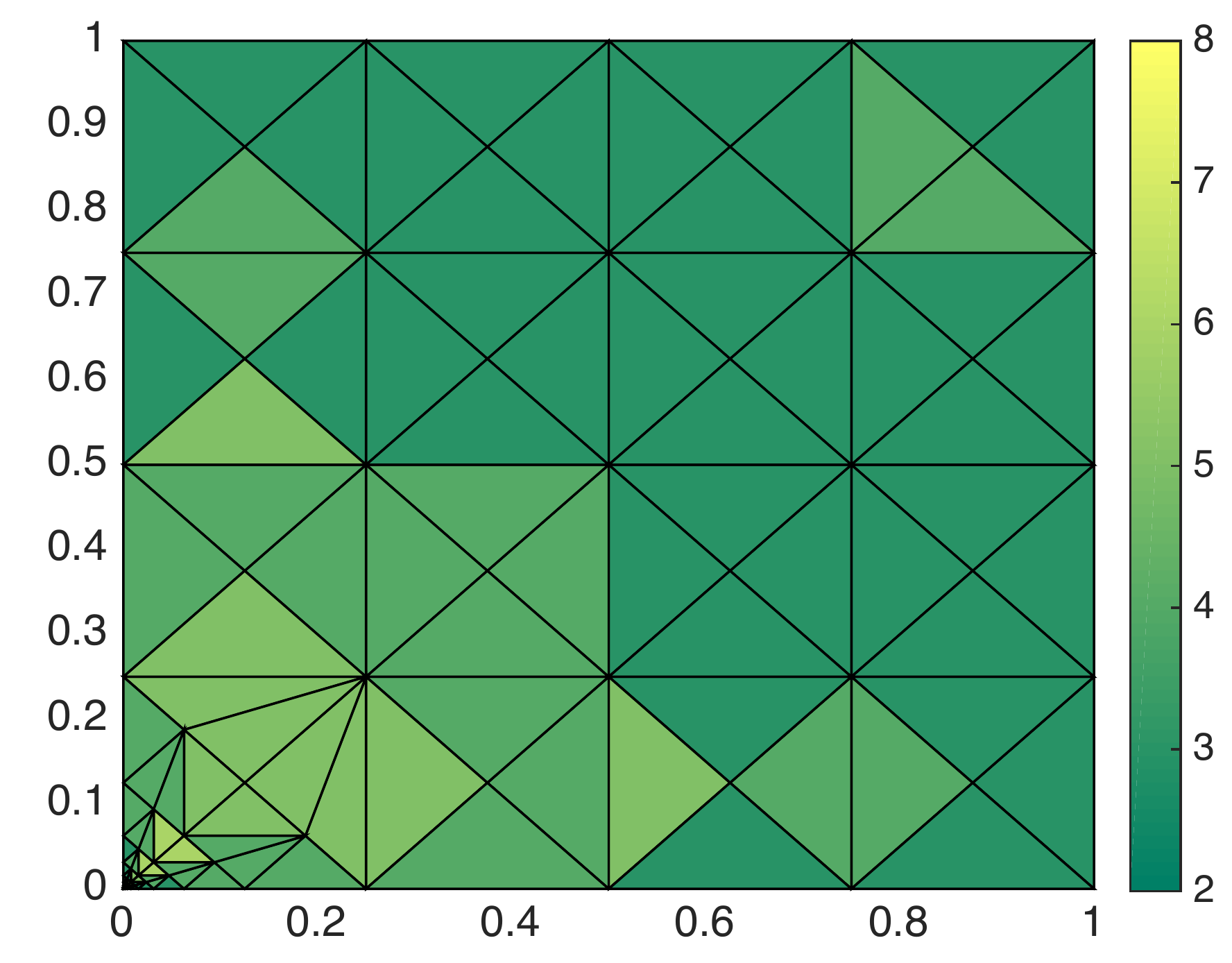} \\
(a) \\
\includegraphics[scale=0.42]{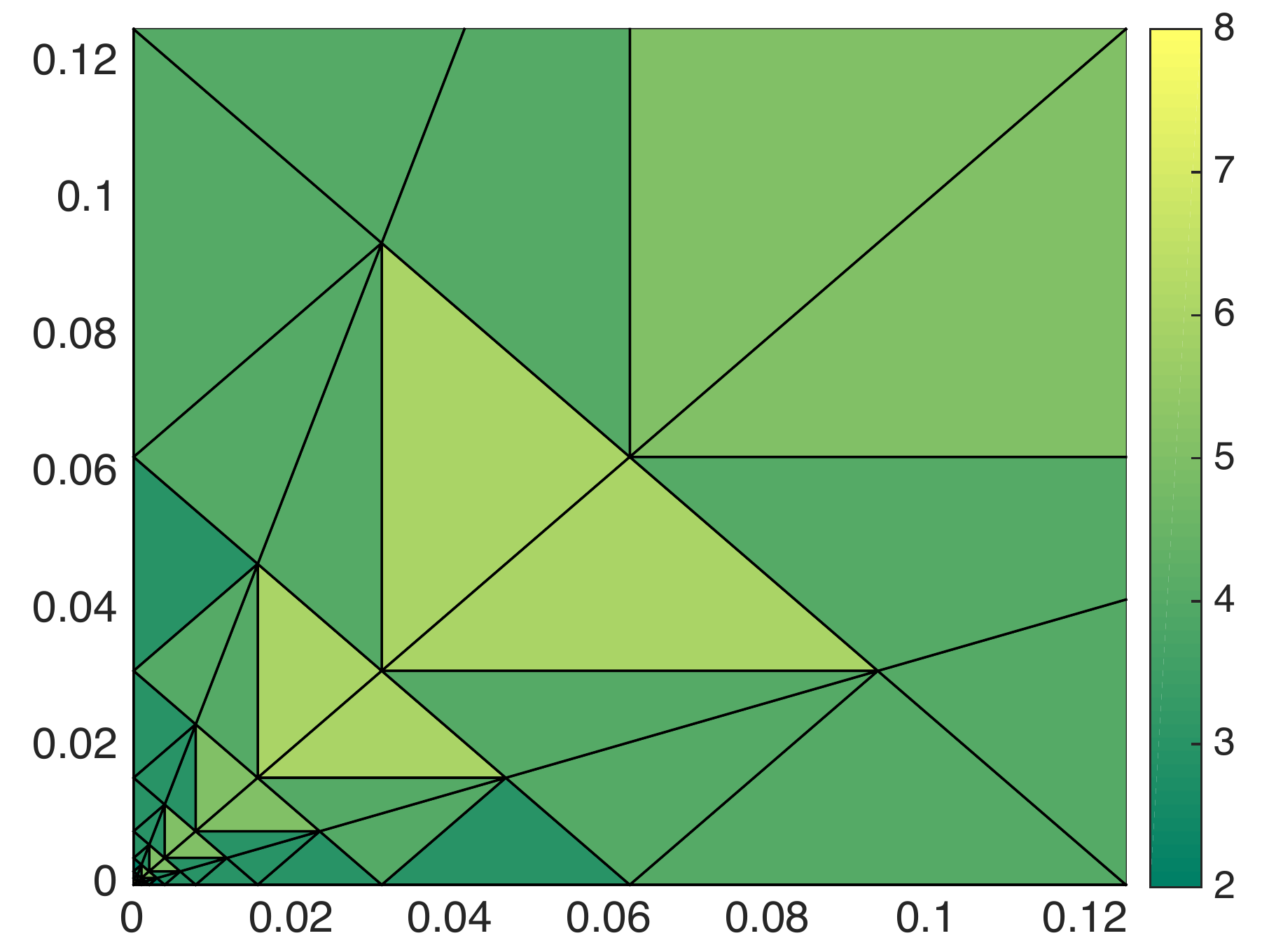} \\
(b) \\
%\end{tabular}
\end{center}
\caption{Example 3. (a) $hp$--Mesh distribution after 23 adaptive refinements;
(b) Zoom of (a).}
\label{ex3-fig2}
\end{figure}

\section{Conclusions}\label{sc:concl}

In this article, we have proposed a novel $hp$--adaptive refinement procedure for application
to the finite element approximation of convex variational problems. In particular, the underlying adaptive algorithm exploits a competitive refinement technique which 
seeks to maximise the decrease in the elemental contribution to the total energy
based on employing local $p$-- and $hp$--enrichments of the finite element space. Whilst our approach
has been successfully applied to a range of second--order quasilinear problems in both one-- and two--dimensions, we emphasise that it is immediately extensible to more general variational-based PDE problems. Future work will be concerned with exploiting anisotropic $hp$--mesh adaptation.

\bibliographystyle{amsplain}
\bibliography{literature}

\end{document}